\documentclass[12pt,reqno]{amsart}
\usepackage{amssymb}
\usepackage{amsthm}
\usepackage{bm}
\usepackage{latexsym}
\usepackage{amsmath}
\usepackage{eufrak}
\usepackage{mathrsfs}
\usepackage{caption}
\usepackage{amscd}
\usepackage[all,cmtip]{xy}
\usepackage[usenames]{color}
\usepackage{graphicx}
\usepackage{color}
\usepackage{amscd}
\usepackage{float}
\usepackage{graphics}
\usepackage{tikz}
\usepackage{tikz-cd}
\usepackage{comment}
\usepackage{xspace}
\usepackage{mathtools}

\usepackage{amssymb}
\usepackage{amsthm}
\usepackage{graphicx}
\usepackage{subfig}
\usepackage{enumerate}
\usepackage{hyperref}

\usetikzlibrary{patterns,decorations.pathreplacing}
\usepackage{marginnote}

\theoremstyle{plain}

\newtheorem{theorem}{Theorem}[section]
\newtheorem{proposition}[theorem]{Proposition}
\newtheorem{lemma}[theorem]{Lemma}
\newtheorem{corollary}[theorem]{Corollary}
\newtheorem{conjecture}[theorem]{Conjecture}

\newtheorem{question}[theorem]{Question}

\theoremstyle{definition}

\newcommand{\appsection}[1]{\let\oldthesection\thesection
\renewcommand{\thesection}{Appendix \oldthesection}
\section{#1}\let\thesection\oldthesection}

\newtheorem{definition}[theorem]{Definition}

\theoremstyle{remark}

\newtheorem{example}[theorem]{Example}

\def\D{{\mathbb{D}}}

\def\Z{{\mathbb{Z}}}

\def\Q{{\mathbb{Q}}}

\def\P{{\mathbb{P}}}

\def\O{{\mathcal{O}}}

\def\X{{\mathcal{X}}}

\begin{document}
\title{On wormholes in the moduli space of surfaces}
\author[Giancarlo Urz\'ua]{Giancarlo Urz\'ua}
\email{urzua@mat.uc.cl}
\address{Facultad de Matem\'aticas, Pontificia Universidad Cat\'olica de Chile, Campus San Joaqu\'in, Avenida Vicu\~na Mackenna 4860, Santiago, Chile.}

\author[Nicol\'as Vilches]{Nicol\'as Vilches}
\email{nivilches@math.columbia.edu}
\address{Department of Mathematics, Columbia University, 2990 Broadway, New York, NY 10027, USA.}


\begin{abstract} 
We study a certain wormholing phenomenon that takes place in the Koll\'ar--Shepherd-Barron--Alexeev (KSBA) compactification of the moduli space of surfaces of general type. It occurs because of the appearance of particular extremal P-resolutions in surfaces on the KSBA boundary. We state a general wormhole conjecture, and we prove it for a wide range of cases. At the end, we discuss some topological properties and open questions. 
\end{abstract}

\subjclass[2010]{14J10, 14E30, 14J29, 14J17}

\keywords{Moduli space of surfaces of general type, KSBA compactification, Minimal model program, Wahl singularity, continued fractions}

\maketitle

\tableofcontents

\section{Introduction} \label{s1}

Since the breakthrough construction of simply connected Campedelli surfaces by Y. Lee and J. Park in \cite{LP07}, there have been several results on various aspects of one parameter $\Q$-Gorenstein degenerations of surfaces, see e.g. \cite{PPS09a}, \cite{PPS09b}, \cite{LN13}, \cite{SU16}, \cite{HTU13}, \cite{U16a}, \cite{U16b}, \cite{RTU17}, \cite{LN18}, \cite{PPSU18}, \cite{CU18}, \cite{EU18}. One of those aspects has been the study of KSBA surfaces with only Wahl singularities which admit $\Q$-Gorenstein smoothings into surfaces of general type. These smoothings could be seen as punctured disks $\D^{\times}$ on the moduli space of surfaces of general type $M_{K^2,\chi}$, which are completed in the KSBA compactification $\overline{M}_{K^2,\chi}$ with a normal projective surface $X$ with only Wahl singularities and $K_X$ ample. (Here of course $K_X^2=K^2$ and $\chi(\O_X)=\chi$.) In this way, we have a $\Q$-Gorenstein smoothing $$(X \subset \X) \to (0 \in \D),$$ where $\D=\overline{\D^{\times}}$. Nowadays there are many examples of such situations in the literature, most of them constructed abstractly, starting with the original work \cite{LP07}.

P-resolutions were introduced by Koll\'ar--Shepherd-Barron to classify deformations of quotient singularities \cite[Section 3]{KSB88}. The smallest ones over cyclic quotient singularities, which are called extremal P-resolutions (see Definition \ref{extremalP}), play a key role for us in the following sense. Sometimes a surface $X$ as above has an embedded extremal P-resolution, which in addition admits another extremal P-resolution over the same cyclic quotient singularity. One performs the corresponding ``extremal P-resolution surgery" on $X$ to obtain another normal projective surface $X'$ with only Wahl singularities. Let us assume that $X'$ admits a $\Q$-Gorenstein smoothing $(X' \subset \X') \to (0 \in \D)$. (This automatically holds under a cohomological condition on $X$, which is used in all Lee--Park type of surfaces.) If $K_{X'}$ is ample, then one can easily show that $X$ and $X'$ live in the same $\overline{M}_{K^2,\chi}$. If $K_{X'}$ is only nef, then the canonical model of $X'$ and $X$ belong to the same $\overline{M}_{K^2,\chi}$ as well. But if $K_{X'}$ is not nef, then one needs to run the minimal model program (MMP) on the $3$-fold family $(X' \subset \X') \to (0 \in \D)$ to find the \textit{KSBA replacement} (in case that the smooth fiber is of general type), this is, the canonical model of a new family $(X'' \subset \X'') \to (0 \in \D)$ such that $K_{X''}$ is nef and $K_{X''}^2>0$. This MMP requires flips and/or divisorial contractions as studied in \cite{HTU13} (see also \cite{U16b}). If $(X' \subset \X') \to (0 \in \D)$ has a minimal model (i.e. canonical class becomes nef) and MMP only requires flips, then the KSBA replacement is again on the same $\overline{M}_{K^2,\chi}$.

\begin{conjecture}[Wormhole conjecture] The MMP requires only flips and gives a minimal model. The KSBA replacement of $(X' \subset \X') \to (0 \in \D)$ lives on the same moduli space as the original $(X \subset \X) \to (0 \in \D)$.
\end{conjecture}

It is not clear if the smooth surfaces in the wormhole are deformation equivalent, i.e. they belong to the same connected component of $M_{K^2,\chi}$. For example, M. Reid conjectures that there is one component for torsion free Godeaux surfaces, and we do have wormholes there by means of Lee--Park type of examples, which we do not know how to connect. On the other hand, wormholes applied to elliptic surfaces may change the topology. We show examples of that in Section \ref{s6}. 

In this paper, we prove the wormhole conjecture for a wide range of cases. We point out that a fixed extremal P-resolution in $X$ can produce at most one wormhole, because in \cite[Section 4]{HTU13} it is proved that a cyclic quotient singularity can admit at most two extremal P-resolutions. Additionally, when that happens, both share the same $\delta$ invariant, this is, for these two surfaces we have that the intersection of the exceptional curve with the canonical class times the indices of the singularities is the same. In this paper, we give simplified and new proofs of both of these facts.

We now state the main theorems, which will imply positive evidence for the wormhole conjecture as a corollary. For all the definitions we refer to Section \ref{s2} and Section \ref{s3}.

\begin{theorem}
Let $Y$ be a nonrational normal projective surface with one cyclic quotient singularity $(Q \in Y)$, which is smooth everywhere else. Assume that $Q$ admits two extremal P-resolutions $f_i^+ \colon (C_i \subseteq X_i) \to (Q \in Y), i=1, 2$, so that the following is satisfied:

\begin{itemize}
\item The strict transform in the minimal resolution of $X_2$ of the exceptional curve $C_2$ for the extremal P-resolution in $X_2$ is a $\P^1$ with self-intersection $-1$. 
\item The canonical class $K_{X_1}$ is nef.
\item Both surfaces $X_i$ admit $\Q$-Gorenstein smoothings $(X_i \subseteq \X_i) \to (0 \in \D)$.
\end{itemize}
Then, we have that $K_{X_2}$ is nef. 
\label{W1W}
\end{theorem}

\begin{theorem}
Let $Y$ be a nonrational normal projective surface with one cyclic quotient singularity $(Q \in Y)$, which is smooth everywhere else. Assume that $Q$ admits two extremal P-resolutions $f_i^+ \colon (C_i \subseteq X_i) \to (Q \in Y), i=1, 2$, so that the following is satisfied:

\begin{itemize}
\item The strict transform in the minimal resolution of $X_2$ of the exceptional curve $C_2$ for the extremal P-resolution in $X_2$ is a $\P^1$ with self-intersection $-2$. 
\item The extremal P-resolution in $X_2$ has only one singularity.
\item The canonical class $K_{X_1}$ is nef.
\item Both surfaces $X_i$ admit $\Q$-Gorenstein smoothings $(X_i \subseteq \X_i) \to (0 \in \D)$.
\end{itemize}
Then, we only need flips to run MMP on $(X_2 \subset \X_2) \to (0 \in \D)$.
\label{2W}
\end{theorem}

We can show via an explicit example that one might indeed need to perform flips in a situation as in Theorem \ref{2W} (see Section \ref{s3}). Finally, in Section \ref{s6} we briefly show and discuss certain topological aspects of wormholes, ending with some open questions and with what is left to prove the wormhole conjecture. We also present a (combinatorial) potential counterexample.  

\begin{corollary}
Let $X$ be a normal projective surface with only Wahl singularities and $K_{X}$ ample. We assume:

\begin{itemize}
\item The surface $X$ is not rational.

\item There is an embedded extremal P-resolution in $X$ such that its contraction $(C \subset X) \to (Q \in Y)$ admits another extremal P-resolution $(C' \subset X') \to (Q \in Y)$ as in Theorem \ref{W1W} or Theorem \ref{2W}.

\item  The cohomology group $H^2\big(\widetilde{X},T_{\widetilde{X}}^0(-\log(E+\widetilde{C}) )\big)$ vanishes, where $\widetilde{X} \to X$ is the minimal resolution of $X$, $E$ is the exceptional divisor, and $\widetilde{C}$ is the strict transform of $C$. Hence, there are $\Q$-Gorenstein smoothings  $(X \subseteq \X) \to (0 \in \D)$ and  $(X' \subseteq \X') \to (0 \in \D)$.

\end{itemize}
Then, the KSBA replacement of $(X' \subset \X') \to (0 \in \D)$ lives on the same moduli space as the original $(X \subset \X) \to (0 \in \D)$.
\label{PartialWC}
\end{corollary}

\subsubsection*{Notation and conventions}

\noindent  

\begin{itemize}
\item A $(-m)$-curve is a curve $\Gamma$ isomorphic to $\P^1$ with $\Gamma^2=-m$.

\item On a normal surface we use the intersection theory for Weil divisors defined by Mumford in \cite[II (b)]{M61}.

\item If $\phi \colon X \to W$ is a birational morphism, then exc$(\phi)$ is the exceptional divisor.

\item A KSBA surface in this paper is a normal projective surface with log-canonical singularities and ample canonical class \cite{KSB88}.

\item Under a birational map, we may keep the notation for a curve and its strict transform. 

\item For a normal projective surface $Z$, the tangent sheaf is denoted by $T_Z^0:=\mathcal{H}om_{\O_Z}(\Omega_Z^1,\O_Z)$. If $Z$ is not singular and $D$ is a simple normal crossings divisor on $Z$, then $T_Z^0$ is the usual rank $2$ tangent bundle and $T_Z^0(-\log(D))$ is the dual of the rank $2$ vector bundle of differentials with simple poles along $D$.
\end{itemize}

\subsubsection*{Acknowledgments}

The results of this article are mainly based on the master's thesis \cite{V20}. We would like to thank Jonny Evans for useful discussions and comments, and to the referees for constructive and useful suggestions. The first-named author was supported by the FONDECYT regular grant 1190066. The second-named author was funded by the ANID scholarship 22190759.

\section{A review on continued fractions and extremal P-resolutions} \label{s2}

\subsection{Continued fractions}

\begin{definition}
Given $a_1, a_2, \dots, a_r$ positive integers, we define the \emph{Hirzebruch-Jung continued fraction} recursively. If $r=1$, then $[a_1] := a_1$. If $r \geq 2$ and $[a_2, \dots, a_r] \neq 0$, then we define $$ [a_1, \dots, a_r] := a_1 - \frac{1}{[a_2, \dots, a_r]}. $$
\end{definition}

Note that not every list of positive integers makes sense as continued fraction, for an example take $[5, 1, 2, 1]$. On the other hand, if $a_i \geq 2$ for every $i$, the continued fraction automatically makes sense, and $[a_1, \dots, a_r]>1$ by induction on $r$. If $0<q<n$ are coprime numbers, then there exists unique $a_i \geq 2$ such that $$[a_1, \dots, a_r] = \frac{n}{q}.$$

To analyze these continued fractions, given $a_1, \dots, a_r$, we define sequences $p_0=1, p_1=a_1, q_0=0, q_1=1$, and for $2 \leq i \leq r$, $$ p_i = a_ip_{i-1}-p_{i-2}, \quad q_i = a_iq_{i-1} - q_{i-1}.$$ Inductively, one can show that
$$ \begin{pmatrix} a_1 & -1 \\ 1 & 0 \end{pmatrix} \cdot \dots \cdot \begin{pmatrix} a_i & -1 \\ 1 & 0 \end{pmatrix} = \begin{pmatrix} p_i & -p_{i-1} \\ q_i & -q_{i-1} \end{pmatrix}, $$
and also $\frac{p_i}{q_i} = [a_1, \dots, a_i]$ for every $1 \leq i \leq r$. We say that $\{a_1, \dots, a_r\}$ is \emph{admissible} if $p_i >0$ for $i <r$. A sequence is admissible if and only if the matrix
\begin{equation} \label{autoint}
\begin{bmatrix} -a_1 & 1 &  &  &  \\ 1 & -a_2 & 1 &  &  \\  & 1 & -a_3 &  &  \\  &  &  & \ddots & 1 \\  &  &  & 1 & -a_r \end{bmatrix}
\end{equation}
is seminegative definite of rank $\geq r-1$ (see e.g. \cite{OW77}). Note that if $a_i \geq 2$ for all $i$, then the sequence is admissible. If some $a_i$ is $1$ and $r \geq 2$, then 
\begin{align*}
&\{a_2-1, \dots, a_r\}, & &i=1; \\
&\{ a_1, \dots, a_{i-1}-1, a_{i+1}-1, \dots, a_r \}, & &1 \leq i \leq r-1; \\
&\{ a_1, \dots, a_{r-1}-1 \}, & &i=r.
\end{align*}
are also admissible. We call this procedure a \emph{blow-down}. If the original fraction was $\frac{n}{q}$, then the new one is $\frac{n}{q'}$ with $q' \equiv q \pmod{n}$.

Given an admissible continued fraction $[a_1, \dots, a_r]$, after blowing-down every possible entry, we may get two different results, according to the rank of the matrix \eqref{autoint}. If its rank is $r$, then we get either $[1]$ or a continued fraction $[b_1, \dots, b_s]$ with $b_j \geq 2$ for every $1 \leq j \leq s$. Otherwise, we get $[1, 1]$ as a final fraction.

We define \emph{zero continued fraction} as an admissible continued fraction $[a_1, \dots, a_r]$ whose value is equal to zero. Equivalent, the rank of its matrix \eqref{autoint} is $r-1$.

Given a fraction $[a_1, \dots, a_r]=\frac{n}{q}$ with $a_i \geq 2$ and $0 < q <n$ coprime, the \emph{dual fraction} is
$$ \frac{n}{n-q} = [b_1, \dots, b_s], $$
with $b_j \geq 2$ for all $j$. We have a visual way to compute them, c.f. \cite{Rie74}. Draw $a_1-1$ dots horizontally. Under the rightmost one, draw another horizontal line of $a_2-1$ dots, and repeat. For instance, if we apply this to $\frac{19}{11}=[2, 4, 3]$, then we obtain Figure \ref{fig:puntitos}. 

\begin{figure}[htbp]
    \centering
    \begin{tikzpicture}
        \draw[fill=black] (0,1) circle (0.05);
        \draw[fill=black] (0,0.5) circle (0.05);
        \draw[fill=black] (0.5,0.5) circle (0.05);
        \draw[fill=black] (1,0.5) circle (0.05);
        \draw[fill=black] (1,0) circle (0.05);
        \draw[fill=black] (1.5,0) circle (0.05);
    \end{tikzpicture}
    \caption{Dot diagram for $[2, 4, 3]$.}
    \label{fig:puntitos}
\end{figure}
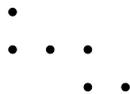

Then, we have $b_1-1$ dots on the first column, $b_2-1$ on the second one, and so on. This shows that $\frac{19}{8}=[3, 2, 3, 2]$. 

Suppose that $ [a_1, \dots, a_r] = \frac{n}{q}$ with $0<q<n$ are coprime and $a_i \geq 2$. One can prove that $$ \begin{pmatrix} a_1 & -1 \\ 1 & 0 \end{pmatrix} \cdot \dots \cdot \begin{pmatrix} a_i & -1 \\ 1 & 0 \end{pmatrix} = \begin{pmatrix} n & -q' \\ q & \frac{1-qq'}{n} \end{pmatrix}, $$
where $q'$ is the inverse of $q$ modulo $n$, since every matrix on the left has determinant 1. Thus, if $[b_1, \dots, b_s] = \frac{n}{n-q}$ is the unique continued fraction with $b_j \geq 2$, then $$ [a_1, \dots, a_r, 1, b_1, \dots, b_s] = 0. $$

\subsection{Zero continued fractions}

Now we will focus on zero continued fractions, following \cite{Ste91}. Consider a zero continued fraction $[a_1, \dots, a_r]$. Blowing down every possible $1$ until the length is $2$, we get $[1, 1]$. Reversing the process, every zero continued fraction can be obtained from $[1, 1]$ through the \emph{blow-ups}
$$ \{a_1, \dots, a_r\} \mapsto \begin{cases} \{1, a_1+1, a_2, \dots, a_r\}, \\ \{ a_1, \dots, a_{i-2}, a_{i-1}+1, 1, a_i+1, a_{i+1}, \dots, a_r\}, \\ \{ a_1, \dots, a_{r-1}, a_r+1, 1 \}. \end{cases} $$

We will show an explicit bijection with triangulation of polygons. A triangulation of a convex polygon $P_0P_1 \dots P_r$ is given by drawing some non-intersecting diagonals on it which divide the polygon into triangles. For a fixed triangulation, we define $v_i$ as the number of triangles which have $P_i$ as one of its vertices. Note that
\begin{equation} \label{sumatriang}
\sum_{i=0}^r v_i= 3(r-1).
\end{equation}

Using induction, one can show  that $[a_1, \dots, a_r]$ is a zero continued fraction if and only if there exists a triangulation of $P_0P_1\dots P_r$ such that $v_i=a_i$ for every $1 \leq i \leq r$. In this way, the number of zero continued fractions of length $r$ is the Catalan number $\frac{1}{r}\binom{2(r-1)}{r-1}$. Also by induction, every triangulation has at least two $v_i$ equal to $1$. They cannot be adjacent unless $r=2$.

\subsection{Cyclic quotient singularities}

\begin{definition}
Given coprime numbers $0<q<n$, the \emph{cyclic quotient singularity} $\frac{1}{n}(1, q)$ is the germ at $0$ of the quotient of $\mathbb{C}^2$ by the action $\zeta \cdot (x, y)=(\zeta x, \zeta^q y)$, where $\zeta$ is a primitive $n$-root of unity.
\end{definition}

The minimal resolution of $X=\frac{1}{n}(1, q)$ can be recover from the continued fraction of $\frac{n}{q}$. If $\frac{n}{q}=[e_1, \dots, e_r]$ with $e_i \geq 2$ and $\sigma\colon \tilde{X} \to X$ is the minimal resolution, the exceptional divisor consists of a chain of $r$ nonsingular rational curves $E_1, \dots, E_r$ with $E_i^2=-e_i$. This is pictured in Figure \ref{fig:resminimalHJ}.

\begin{figure}[htbp]
    \centering
    \begin{tikzpicture}
        \draw (1.5,0.5) ellipse (1.5 and 0.5);
        \draw (1.5,2.5) ellipse (1.5 and 0.5);
        \draw[fill=black] (1.5,0.5) circle (0.05);
        \draw (0.6,2.2) -- (1.2,2.8);
        \draw (1.0,2.8) -- (1.6,2.2);
        \draw[style=dotted] (1.4,2.2) -- (2.0,2.8);
        \draw (1.8,2.8) -- (2.4,2.2);
        \draw (1.5,1.8) -- (1.5,1.2);
        \draw (1.6,1.3) -- (1.5,1.2) -- (1.4,1.3);
        
        \begin{tiny}
        \draw (1.8,0.5) node {$Q$};
        \draw (3.3,0.5) node {$X$};
        \draw (3.3,2.5) node {$\tilde{X}$};
        \draw (0.75,2.7) node {$E_1$};
        \draw (1.45,2.7) node {$E_2$};
        \draw (2.25,2.7) node {$E_r$};
        \end{tiny}
    \end{tikzpicture}
    \caption{Minimal resolution of $\frac{1}{n}(1, q)$.}
    \label{fig:resminimalHJ}
\end{figure}
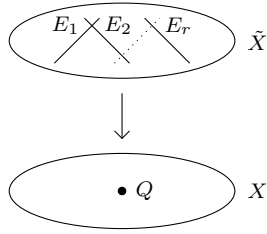

Note that if we do a blow-up at the intersection of $E_i$ and $E_{i+1}$, we get a new chain $E_1, \dots, E_i, F, E_{i+1}, E_r$ of self-intersections $E_i^2=-(e_i+1), E_{i+1}^2=-(e_{i+1}+1), F^2=-1$. A similar remark can be made for blow-downs. This justifies the terminology \emph{blow-down} for continued fractions. We note that we can compare the canonical divisor on $X$ and $\tilde{X}$ as follows 
\begin{equation} \label{ecdiscrepancias}
    K_{\tilde{X}} \equiv \sigma^\ast K_X + \sum_{i=1}^r k_i E_i,
\end{equation}
where $-1<k_i\leq 0$ are the discrepancies of $E_i$.

\begin{definition}
A \emph{Wahl singularity} is a cyclic quotient singularity $\frac{1}{m^2}(1, ma-1)$, where $0<a<m$ are coprime numbers.
\end{definition}

An alternative description can be made by looking at the continued fraction (see \cite[Lemma 3.11]{KSB88}). Every Wahl singularity arises from $[4]$ by applying the operations
\begin{equation} \label{procWahl}
[a_1, \dots, a_r] \mapsto \begin{cases} [2, a_1, \dots, a_{r-1}, a_r+1] \\ [a_1+1, a_2, \dots, a_r, 2]. \end{cases}    
\end{equation}
From this algorithm and by induction on $r$, it is clear that every Wahl singularity $\frac{m^2}{ma-1}=[a_1, \dots, a_r]$ satisfies $\sum_{i=1}^r a_i=3r+1$.

Let $[a_1, \dots, a_r]$ be a Wahl continued fraction. We define integers $\delta_1, \dots, \delta_r$ in the following inductive way. If $r=1$ then $\delta_1:=1$. If we already defined $\delta_1, \dots, \delta_r$ for $[a_1, \dots, a_r]$, then we assign
\begin{align*}
\delta_1, \dots, \delta_r, \delta_1+\delta_r &\text{ to } [a_1+1, \dots, a_r,2] \\
\delta_1+\delta_r, \delta_1, \dots, \delta_r &\text{ to } [2, a_1, \dots, a_r+1].
\end{align*}
These numbers compute the discrepancies in Equation \eqref{ecdiscrepancias}. If $\frac{m^2}{ma-1}=[a_1, \dots, a_r]$ has numbers $\delta_1, \dots, \delta_r$, then
\begin{equation} \label{ecdiscwahl}
    K_{\tilde{X}} \equiv \sigma^\ast K_X + \sum_{i=1}^r \left( -1+\frac{\delta_i}{\delta_1+\delta_r} \right) E_i.
\end{equation}

This gives us an explicit control on discrepancies, which will be used to bound them later in this paper. 

\subsection{Extremal P-resolutions and wormhole singularities}

For the study of the components of the deformation space of quotient singularities, Koll\'ar--Shepherd-Barron introduced P-resolutions in \cite[Section 3]{KSB88}. We only need a particular class of them.

\begin{definition}
Let $0<\Omega<\Delta$ be coprime integers, and let $(Q \in Y)$ be a cyclic quotient singularity $\frac{1}{\Delta}(1, \Omega)$. An \emph{extremal P-resolution} of $(Q \in Y)$ is a partial resolution $f_0^+\colon (C^+ \subset X^+) \to (Q \in Y)$, such that $X^+$ has only Wahl singularities, there is one exceptional curve $C^+$ and isomorphic to $\mathbb{P}^1$, and $K_{X^+}$ is relatively ample.
\label{extremalP}
\end{definition}

Following \cite[\textsection 4]{HTU13}, the surface $X^+$ has at most two Wahl singularities $\frac{1}{m_i^2}(1, m_ia_i-1)$. If we have smooth points, then we set $m_i=a_i=1$. If their associated continued fractions are given by
$$ \frac{m_1^2}{m_1a_1-1}=[e_1, \dots, e_{r_1}], \quad \frac{m_2^2}{m_2a_2-1}=[f_1, \dots, f_{r_2}], $$
and $(C^+)^2=-c$ on the minimal resolution of $X^+$, then
$$ \frac{\Delta}{\Omega}= [f_{r_2}, \dots, f_1, c, e_1, \dots, e_{r_1}]. $$ We denote the extremal P-resolution as $[f_{r_2}, \dots, f_1] - c - [e_1, \dots, e_{r_1}]$. The intersection $K^+ \cdot C^+$ can be computed as $\frac{\delta}{m_1m_2}$, where $\delta=cm_1m_2-m_1a_2-m_2a_1$. The self-intersection $-c$ of $C^+$ can be computed in terms of the continued fraction of $\frac{\Delta}{\Omega}$. 

\begin{theorem} \label{calcularc}
Consider a cyclic quotient singularity $Y=\frac{1}{\Delta}(1, \Omega)$, with $\frac{\Delta}{\Omega}=[b_1, \dots, b_r]$. Suppose that we have an extremal P-resolution $(C^+ \subset X^+)$ over $\frac{1}{\Delta}(1, \Omega)$ with $l$ singularities ($l=0, 1$ or $2$). Then, the self-intersection of the exceptional curve $C^+$ on the minimal resolution of $X^+$ is $-(\sum_{i=1}^r b_i-3r+3-l)$.
\end{theorem}

As a direct consequence, note that if $\sum_{i=1}^r b_i<3r$, there are no extremal P-resolutions. If $\sum_{i=1}^r b_i=3r$, then $c$ can be $-1$ (if there are two singularties) or $-2$ (if there is only one singularity) or $-3$ (if $l=0$), and so on.

\begin{proof}
If $l=0$, then $r=1$ and the result is trivially true. Suppose that $l=2$; the proof for $l=1$ is similar. Consider the extremal P-resolution 
$$ [f_{r_2}, \dots, f_1]-c-[e_1, \dots, e_{r_1}]. $$
Note that $\sum_{i=1}^{r_1} e_i+c+\sum_{j=1}^{r_2} f_j =3(r_1+r_2)+c+2$, since we have two Wahl singularities, and so $\sum_{i=1}^{r_1} e_i=3r_1+1$ and $\sum_{j=1}^{r_2} f_j=3r_2+1$. 

From \cite[Lemma 3.13, Lemma 3.14]{KSB88} we know that, from the minimal resolution of $Y$, one has to blow-up only at the intersection points of exceptional curves to obtain the minimal resolution of the extremal P-resolution. In this way, the sum of self-intersections of exceptional curves plus three times the amount of them remains constant at each blow-up (since at every blow-up we subtract $3$ to the sum of self-intersections, and we add $1$ to the amount of curves). This shows that
$$ \sum_{i=1}^r b_i-3r = \left(\sum_{i=1}^{r_1} e_i+c+\sum_{j=1}^{r_2} f_j\right) - 3(r_1+r_2+1).$$ It follows that $\sum_{i=1}^r b_i-3r=c-1$.
\end{proof}

Given a coprime pair $0<\Omega<\Delta$, one can find all possible extremal P-resolutions by looking at the dual fraction $\frac{\Delta}{\Delta-\Omega}$. More precisely, we have the following result (see \cite[Prop 4.1.]{HTU13}).

\begin{proposition} \label{Presdesdedual}
If $\frac{\Delta}{\Delta-\Omega}=[c_1, \dots, c_s]$, then there is a bijection between extremal P-resolutions and pairs $1 \leq \alpha<\beta \leq s$ such that
\begin{equation} \label{restardos}
[c_1, \dots, c_{\alpha-1}, c_\alpha-1, c_{\alpha+1}, \dots, c_{\beta-1}, c_\beta-1, c_{\beta+1}, \dots, c_s]=0. 
\end{equation}
Moreover, the $a_i,m_i$ and $\delta$ of the corresponding extremal P-resolution (see right after Definition \ref{extremalP}) can be computed as: $\frac{m_2}{a_2}=[c_1, \dots, c_{\alpha-1}]$, $\frac{m_1}{a_1}=[c_s, \dots, c_{\beta+1}]$ (if $\alpha=1$ or $\beta=s$, the associated points are smooth), and  $\frac{\delta}{\varepsilon}=[c_{\alpha+1}, \dots, c_{\beta-1}]$, where $0<\varepsilon<\delta$ (or $\delta=1$ if $\alpha+1=\beta$).
\end{proposition}

It will be useful to denote the expression in Equation \eqref{restardos} with two bars as $$ [c_1, \dots, \overline{c}_\alpha, \dots, \overline{c}_\beta, \dots, c_s]. $$ Moreover, if it admits a second extremal P-resolution, then we will denote it with two underlines. For instance, if $\Delta=36, \Omega=13$, then we write
$$ \frac{36}{36-13}=[2, \underline{3}, \overline{2}, \underline{2}, \overline{4}], $$ and so we know it admits two extremal P-resolutions, and we know how to obtain them. In this example, $[2, 3, \overline{2}, 2, \overline{4}]$ is associated to the extremal P-resolution $[3, 5, 2]-2$, and $[2, \overline{3}, 2, \overline{2}, 4]$ corresponds to $[4]-1-[6, 2, 2]$.

\begin{definition} 
As in \cite[Section 4]{HTU13}, a sequence $\{a_1, \dots, a_r\}, a_i >1$ is \emph{of WW type} if there exists $1\leq \alpha<\beta\leq r$ such that
$$ [a_1, \dots, \overline{a}_\alpha, \dots, \overline{a}_\beta, \dots, a_r]=0. $$

A \emph{wormhole} singularity is a cyclic quotient singularity $\frac{1}{\Delta}(1, \Omega)$ which admits at least two extremal P-resolutions. Equivalently, the continued fraction of $\frac{\Delta}{\Delta-\Omega}$ is of WW type by means of at least two pairs $(\alpha, \beta), (\alpha', \beta')$.
\end{definition}

As a consequence of Theorem \ref{atmosttwo}, we will have that a wormhole singularity admits precisely two extremal P-resolutions.

If $\{a_1, \dots, a_r\}$ is a sequence of WW type, there is a triangulation of a polygon $P_0P_1\dots P_r$ such that $v_i=a_i$. Thus, we define $a_0:=v_0$. Note that by Equation \eqref{sumatriang} we have $a_0=3r-1-\sum_{i=1}^r a_i$, so it does not depend on the pair $(\alpha, \beta)$. Note also that $a_0$ may be $1$. 

Therefore, we have two cases: (A) $a_0>1$ or (B) $a_0=1$ (as in \cite[\textsection 4.2]{HTU13}). We will focus on proving statements for the case (A), since from that proof we will deduce the case (B) as a consequence. The main idea is that we can ``remove'' the vertex $P_0$ from the polygon $P_0P_1\dots P_r$, and repeat until all entries are greater than $1$. 

Our next goal is to give a simplified proof of \cite[Thm. 4.3.]{HTU13}, and a new proof of Theorem \cite[Thm 4.4.]{HTU13}. 

\begin{theorem}[{\cite[Thm 4.3.]{HTU13}}] \label{atmosttwo}
A cyclic quotient singularity has at most two distinct extremal P-resolutions.
\end{theorem}

\begin{theorem}[{\cite[Thm 4.4.]{HTU13}}] \label{samedelta}
If a cyclic quotient singularity admits two extremal P-resolutions, then the $\delta$’s are equal.
\end{theorem}

To prove Theorem \ref{atmosttwo}, note that it suffices to show, by Proposition \ref{Presdesdedual}, that a sequence of WW type $\{a_1, \dots, a_r\}$ admits at most two pairs $(\alpha, \beta)$ such that
$$ [a_1, \dots, a_\alpha-1, \dots, a_\beta-1, \dots, a_r]=0. $$ Let $a_0=3r-1-\sum_{i=1}^r a_i$ as before, and assume that we are in  the case (A), i.e. $a_0>1$. Since the triangulation of $P_0P_1 \dots P_r$ needs to have at least two vertices with $v_i=1$, we must have $a_\alpha-1=a_\beta-1=1$, and thus $a_\alpha=a_\beta=2$.

Note then that $[a_{\alpha+1}, \dots, a_\beta-1, \dots, a_r, a_0, a_1, \dots, a_{\alpha-1}]=0$, since we have a triangulation. A matrix computation shows that
$$ \frac{m}{m-a}=[a_{\beta-1}, \dots, a_{\alpha+1}], \quad \frac{m'}{m'-a'}=[a_{\beta+1}, \dots, a_r, a_0, \dots, a_{\alpha-1}] $$
has $m=m'$ and $a+a'=m$. In this way, we have that
$$ [a_{\alpha+1}, \dots, a_r, a_0, \dots, a_{\alpha-1}]=\frac{m^2}{ma+1} $$
is the dual of a Wahl singularity. All of them are obtained from $[2, 2, 2]$ by applying the same procedure from Equation \eqref{procWahl}, which can be seen as a consequence of Riemenschneider diagrams. See \cite[\textsection 4.2]{HTU13} for another proof of this fact.

Thus, to produce sequences of WW type, we start with $[2, \overline{2}, 2]$, then we apply Equation \eqref{procWahl}, and finally we add a $\overline{2}$ to close the ``cycle". After that we choose one of the entries different from $\overline{2}$, and remove it. That entry will be the $a_0$.

To simplify the proof of Theorem \ref{atmosttwo}, we will use the following sequences of $0$s and $1$s.

\begin{definition}
Given $\{a_0, \dots, a_r\}, a_i \geq 2$, we define its \emph{indicator sequence} as
$$\{ 1, \underbrace{0, \dots, 0}_{a_0-2}, 1, \underbrace{0, \dots, 0}_{a_1-2}, 1, \dots, 1, \underbrace{0, \dots, 0}_{a_r-2} \}$$.
\end{definition}

We think on $\{a_0, \dots, a_r\}$ and its indicator sequence indexed by a cyclic groups. As an example, the indicator sequence of $\{ 2, 3, 4, 2, 3 \}$ is $\{1, 1, 0, 1, 0, 0, 1, 1, 0 \}$. Note that we can completely recover the sequence $\{a_0, \dots, a_r\}$ from the indicator sequence, and so we can study sequences of WW type from its indicator sequence, in the case (A). We also note that, for every $i$, there are two indices $l_i, m_i$ such that $e_{l_i+1}, \dots, e_{m_i-1}$ are all the zeroes induced by $a_i$. The main advantage of this is that it makes more symmetric the procedure from Equation \eqref{procWahl}. We start with $\overline{1, 1}$; then, we add $1$ to one side and $0$ to the other, as follows
\[ \{\overline{1, 1}\} \to \{ 0, \overline{1, 1}, 1 \} \to \{ 1, 0, \overline{1, 1}, 1, 0 \} \to \{ 1, 1, 0, \overline{1, 1}, 1, 0, 0 \} \to \dots \]
We repeat and then add $\overline{1, 1}$ to the end. In particular, all these indicator sequences have an even number of entries, and the numbers 1 which corresponds to the $\overline{2}$ are opposites.

\begin{proof}[Proof of Theorem \ref{atmosttwo}]
We assume $a_0>1$ as before. We consider a sequence of WW type $\{a_1, \dots, a_r\}$, and $\{e_0, \dots, e_{2m-1}\}$ its indicator sequence. Consider $p<q$ such that $[a_1, \dots, \overline{a}_p, \dots, \overline{a}_q, \dots, a_r]=0$. Thus, if $t, t+1$ and $t+m, t+m+1$ are the corresponding indices for $a_p=2$ and $a_q=2$, the construction yields
\begin{equation} \label{sistemae}
e_j=\begin{cases} 2-e_{2t+1-j}, & j = r, r+1, r+m, r+m+1; \\ 1-e_{2t+1-j}, & j \neq r, r+1, r+m, r+m+1. \end{cases}
\end{equation}
Given the indicator sequence, it suffices then to show that there are at most two pairs $\{t, t+m\}$ which makes Equation \eqref{sistemae} true. Since $t$ and $t+m$ gives the same pair, we will define $f_j=(e_j+e_{j+m})/2$, as a sequence indexed by $\Z/m\Z$. The Equation \eqref{sistemae} translates to
\begin{equation} \label{sistemaefe}
f_j = \begin{cases} 2-f_{2t+1-j}, & j = r, r+1; \\ 1-f_{2t+1-j}, & j \neq r, r+1. \end{cases}
\end{equation}

We are going to use the same trick as in \cite[\textsection 4.2]{HTU13}. Fix $\mu$ a primitive $m$-root of unity, and define
\begin{equation} \label{elF}
F = \sum_{j=0}^{m-1} \mu^j f_j.
\end{equation}
Adding Equation \eqref{sistemaefe} multiplied by $\mu^j$ from $j=0$ to $m$, we get
$$ (\mu^t)^2 \cdot \mu \overline{F}-(\mu^t) \cdot (\mu+1) +F=0. $$
For $m>2$, this is an equation of degree at least $1$ on $\mu^t$. Thus, there are at most two valid values of $\mu^t$. Note that $m=2$ happens only for the indicator sequence $\{ 1, 1, 1, 1 \}$, associated to $\{\overline{2}, \underline{2}, \overline{2}, \underline{2}\}$. By the correspondence between sequences of type WW in case (A) and indicator sequences, this shows that there are at most two pairs in this case.

Suppose now that $a_0=1$, this is, we are now on case (B). Note that for every pair $\alpha<\beta$ such that $[a_1, \dots, \overline{a}_\alpha, \dots, \overline{a}_\beta, \dots, a_r]=0$, the corresponding triangulation on $P_0P_1\dots P_r$ must have a triangle $P_0P_1P_r$. We can remove then vertex $P_0$, and look to pairs for the new sequence $a_1-1, a_2, \dots, a_{r-1}, a_r-1$, since it is easy to show that they are in bijection with pairs for the original sequence. Inductively, this reduces case (B) to case (A). 
\end{proof}

\begin{proof}[Proof of Theorem \ref{samedelta}]

We will use Proposition \ref{Presdesdedual}. Consider a sequence $\{a_1, \dots, a_r\}$ with $a_0>1$, i.e. we are in case (A). If $p<q$ is a pair such that $[a_1, \dots, \overline{a}_p, \dots, \overline{a}_q, \dots, a_r]=0$, then $[a_{p+1}, \dots, a_{q-1}]=\frac{\delta}{\varepsilon}$ for some $\varepsilon$. Thus,
$$ [a_{p+1}, \dots, a_r, a_0, \dots, a_{p-1}]=\frac{\delta^2}{\delta \lambda+1}, $$
for some $\lambda<\delta$. Since all entries are $\geq 2$, we can compute
\begin{align*}
& \begin{pmatrix} a_{p+1} & -1 \\ 1 & 0 \end{pmatrix} \dots \begin{pmatrix} a_r & -1 \\ 1 & 0 \end{pmatrix} \begin{pmatrix} a_0 & -1 \\ 1 & 0 \end{pmatrix} \dots \begin{pmatrix} a_{p-1} & -1 \\ 1 & 0 \end{pmatrix} \\
=& \begin{pmatrix} \delta^2 & -\delta(\delta-\lambda)-1 \\ \delta\lambda+1 & -\lambda(\delta-\lambda)-1 \end{pmatrix}.
\end{align*}
Since $a_p=2$, we obtain that the matrix
$$ \begin{pmatrix} a_{p+1} & -1 \\ 1 & 0 \end{pmatrix} \dots \begin{pmatrix} a_r & -1 \\ 1 & 0 \end{pmatrix} \begin{pmatrix} a_0 & -1 \\ 1 & 0 \end{pmatrix} \dots \begin{pmatrix} a_{p-1} & -1 \\ 1 & 0 \end{pmatrix} \begin{pmatrix} a_p & -1 \\ 1 & 0 \end{pmatrix} $$
is
$$  \begin{pmatrix} \delta_i^2+\delta_j\lambda_i-1 & -\delta_i^2 \\ \delta_i\lambda_i+\lambda_i^2+1 & -\delta_i\lambda_i-1 \end{pmatrix}. $$
Its trace is exactly $\delta^2-2$. But recall that the trace of a multiplication is invariant under cyclic permutations of the factors, which shows that
$$ \operatorname{tr} \left( \begin{pmatrix} a_0 & -1 \\ 1 & 0 \end{pmatrix} \dots \begin{pmatrix} a_r & -1 \\ 1 & 0 \end{pmatrix} \right) = \delta^2-2. $$
The left hand side does not depend on the pair $(p, q)$, thus $\delta$ is the same for every pair. This proves Theorem \ref{samedelta} for case (A). 

The case (B) is handled by induction to reduce it to case (A), just like in the proof of Theorem \ref{atmosttwo}. If $a_0=1$, we can blow-down the sequence there. Note that
$$ \begin{pmatrix} a_r & -1 \\ 1 & 0 \end{pmatrix} \begin{pmatrix} 1 & -1 \\ 1 & 0 \end{pmatrix} \begin{pmatrix} a_1 & -1 \\ 1 & 0 \end{pmatrix} = \begin{pmatrix} a_r-1 & -1 \\ 1 & 0 \end{pmatrix} \begin{pmatrix} a_1-1 & -1 \\ 1 & 0 \end{pmatrix}, $$
which shows that the trace remains constant. Also, since the interval $a_{p+1}, \dots, a_{q-1}$ is not affected by blow-downs, this inductively reduces it to case (B). 
\end{proof}

\section{General set-up and the wormhole conjecture} \label{s3}

In this section we will look at singular surfaces together with a smoothing over a smooth analytic curve germ $\D$. This point of view was used in \cite{U16a} under the name of W-surfaces, and it works better for the set-up of the wormhole conjecture. We start by recalling it.

\subsection{W-surfaces and their MMP} \label{s31}

\begin{definition}
A {\em W-surface} is a normal projective surface $X$ together with a proper deformation $(X \subset \X) \to (0 \in \D)$ such that
\begin{enumerate}
\item $X$ has at most Wahl singularities.
\item $\X$ is a normal complex $3$-fold with $K_{\X}$ $\Q$-Cartier.
\item The fiber $X_0$ is reduced and isomorphic to $X$.
\item The fiber $X_t$ is nonsingular for $t\neq 0$.
\end{enumerate}
The W-surface is said to be {\em smooth} if $X$ is nonsingular.
\label{wsurf}
\end{definition}

Various basic properties of W-surfaces are shown in \cite[Section 2]{U16a}. A W-surface $X$ is {\em minimal} if $K_X$ is nef. This is equivalent to $K_{\X}$ nef, as it is shown in \cite[Lemma 2.3]{U16a}. If a  W-surface $X$ is not minimal, then there is an explicit MMP relative to $\D$ which we will review briefly below. The outcomes of this MMP are discussed in \cite[Section 2]{U16a}. We note that invariants such as irregularity, geometric genus, $K
^2$, and Euler topological characteristic are constant for the fibers in a W-surface. An invariant that may not remain constant is the topological fundamental group. We have that $K_X$ ample implies $K_{X_t}$ ample for all $t$, and in this case we may think of a W-surface $X$ as a disk in the KSBA compactification of the moduli space of surfaces of general type with $K^2=K_X^2$ and $\chi=\chi(\O_X)$. 

Let $\sigma \colon \widetilde{X} \to X$ be the minimal resolution of $X$.

\begin{lemma}
Let $X$ be a minimal W-surface such that the minimal resolution of $X$ is ruled. Then $X$ is rational. 
\end{lemma}

\begin{proof}
Assume that $\widetilde{X}$ is ruled but not rational. Then there is a fibration $\widetilde{X} \to C$ whose general fiber is $\P^1$ and $C$ is a nonsingular projective curve of positive genus. Then all curves in the exceptional divisor of $\sigma$ must be contained in fibers. But if $F$ is a general fiber, then $F \cdot K_{\widetilde{X}}=\sigma(F) \cdot K_X$ and by adjunction $F \cdot K_{\widetilde{X}}=-2$, which is contrary to the assumption $K_X$ nef. 
\end{proof}

When a W-surface $X$ has $K_X$ not nef, then there is a smooth rational curve $C$ with $C \cdot K_X <0$. The cases $C^2 \geq 0$ are analyzed in \cite[Section 2]{U16a}, these are not relevant for the present paper. We assume $C^2<0$. Then the W-surface $X$ defines an extremal neighborhood of type mk1A or mk2A, and we need to run MMP on the $3$-fold family $(X \subset \X) \to (0 \in \D)$. Roughly speaking, in case of a flip we will replace $C$ by a $K$-positive curve $C^+ \subset X^+$ obtaining a new family $(X^+ \subset \X^+) \to (0 \in \D)$, where fibers over $t \neq 0$ remain equal to the fibers of the first family. In this way the surface $X^+$ defines a new W-surface. In case of a divisorial contraction, we will have divisor in $\X$ whose restriction to $X$ is $C$, and to any other fiber is a $(-1)$-curve. The contraction of this divisor gives us a new   
family, and the contraction of $C$ produces a Wahl singularity. The new surface is a W-surface. For all details we refer to \cite[Section 2.4]{U16b} (see also \cite[Section 2]{HTU13}, \cite[Section 2]{U16a}). Below we describe the mk1A and mk2A situations on the surface $X$. Let $(C \subset X) \to (Q \in Y)$ be the contraction of $C$. 

\bigskip 
\noindent
\textbf{mk1A:} In this situation $X$ has one Wahl singularity $\frac{1}{m^2}(1,ma-1)$ where $\frac{m^2}{ma-1}=[e_1,\ldots,e_s]$. Let $E_1, \ldots, E_s$ be the corresponding exceptional curves in $\widetilde{X}$, so that $E_j^2=-e_j$. The proper transform $\tilde{C}$ of $C$ is a smooth rational curve intersecting only one $E_i$ transversally at one point. The curve $C$ contracts to $(Q \in Y)$, which is the cyclic quotient singularity $\frac{1}{\Delta}(1,\Omega)$ where $$\frac{\Delta}{\Omega}=[e_1,\ldots,e_{i-1},e_{i}-1,e_{i+1},\ldots,e_s].$$ We will denote this situation as $[e_1,\ldots, \overline{e_i},\ldots, e_s]$. If we write $K_{\widetilde{X}} \equiv \sigma^*(K_X) + \sum_{j=1}^s (-1+\frac{\delta_j}{m}) E_j$ and $\delta:= \delta_i$, we have $$\tilde{C} \cdot K_{\tilde{X}}=-1+\frac{\delta}{m} + C \cdot K_{X} <  C \cdot K_{X} <0,$$ and $\tilde{C}^2<0$ since it is contracted. In particular, the curve $\tilde{C}$ is a $(-1)$-curve. We have $C \cdot K_X = - \frac{\delta}{m}$ and  $C^2=-\frac{\Delta}{m^2}$.

\bigskip 
\noindent
\textbf{mk2A:} In this situation $X$ has two Wahl singularities $\frac{1}{m_j^2}(1,m_ja_j-1)$ for $j=1,2$ where $\frac{m_1^2}{m_1 a_1-1}=[e_1,\ldots,e_{s_1}]$ and $\frac{m_2^2}{m_2a_2-1}=[f_1,\ldots,f_{s_2}]$. Let $E_1,\ldots,E_{s_1}$ and $F_1,\ldots,F_{s_2}$ be the corresponding exceptional curves with $E_j^2=-e_j$ and $F_j^2=-f_j$. The strict transform $\tilde{C}$ of $C$ is a smooth rational curve intersecting only $E_1$ and $F_1$, transversally at one point each. We have that $$\frac{\Delta}{\Omega}=[f_{s_2},\ldots,f_1,1,e_1,\ldots,e_{s_1}]$$ where $(Q \in Y)$ is $\frac{1}{\Delta}(1,\Omega)$. Let $\delta:=m_2 a_1 - m_1(m_2-a_2)$. Then we have $$\tilde{C}\cdot K_{\tilde{X}}= -1+ \frac{\delta}{m_1 m_2}+ C \cdot K_X < C \cdot K_X  <0$$ and we know $\tilde{C}^2<0$. In particular, the curve $\tilde{C}$ is a $(-1)$-curve. We have $C \cdot K_X = - \frac{\delta}{m_1 m_2}$ and  $C^2=-\frac{\Delta}{m_1^2 m_2^2}$.

\bigskip 

To know whether a W-surface $X$ with $C\cdot K_X <0$ and $C^2 <0$ defines a flip or divisorial contraction, we need to run the Mori algorithm from the numerical data of the mk1A or mk2A extremal neighborhood. We refer to \cite[2.4]{U16b} for details, see also \cite{V20} for examples and computer implementation of Mori's algorithm. A summary with relevant properties for us is the following:

\bigskip 
\noindent
\textbf{Divisorial contraction:} In this case the general fiber of the W-surface $X$ contains a $(-1)$-curve which deforms to $C$. This gives us a divisor on the $3$-fold $\mathcal{X}$, which can be contracted to obtain a new W-surface $Y$. The contraction of $C \subset X$ produces a Wahl singularity $(Q \in Y)$. 

\bigskip 
\noindent
\textbf{Flip:} In this case the contraction of $C$ produces a cyclic quotient singularity $\frac{1}{\Delta}(1,\Omega)= (Q \in Y)$. This singularity admits an extremal P-resolution $(C^+ \subset X^+) \to (Q \in Y)$ so that a suitable W-surface $X^+$ is the flip of the W-surface $X$. The general fibers of the W-surfaces $X$ and $X^+$ are isomorphic.

\bigskip 

If a multiple of $K_X$ has sections, then after finitely many flips and/or divisorial contractions of type mk1A and/or mk2A we will obtain a minimal W-surface (see e.g. \cite[Theorem 5.3]{HTU13}). Otherwise, after finitely many flips and/or divisorial contractions of type mk1A and/or mk2A we will end up with either a smooth deformation of a ruled surface, or a degeneration of $\P^2$ (see e.g. \cite{U16a}).  

\subsection{Wormholes} \label{s32}

The following is the set-up for a wormhole. We take a W-surface $X_1$ with $K_{X_1}$ ample, and we assume that $X_1$ has an extremal P-resolution $(C_1 \subset X_1) \to (Q \in Y)$ over a WW singularity $(Q \in Y)$. In this way, the surface $Y$ is obtained from $X_1$ by contracting one smooth rational curve. As cyclic quotient singularities are rational, the irregularity and geometric genus of both surfaces are equal. By the Nakai--Moishezon criterion, the surface $Y$ is a KSBA stable surface but it does not belong to the same moduli space since $K_Y^2=K_{X_1}^2-\nu^2 C_1^2$ for some $\nu \neq 0$ and $C_1^2<0$ (as it is contracted).

Let $E$ be the exceptional (reduced) divisor of the minimal resolution $\widetilde{X_1} \to X_1$, and let $\widetilde{C_1}$ be the strict transform of $C_1$. We also assume that $$H^2(\widetilde{X_1},T_{\widetilde{X_1}}
^0(-\log(E+\widetilde{C_1})))=0.$$ By \cite[Section 2]{LP07}, this condition can be used to prove that there are no-local-to-global obstructions to deform $X_1$ (which in particular shows existence of W-surfaces $X_1$). If $\widetilde{C_1}$ is a $(-1)$-curve, then $H^2(\widetilde{X_1},T_{\widetilde{X_1}}
^0(-\log(E+\widetilde{C_1})))=0$ is the same as $$H^2(\widetilde{X_1},T_{\widetilde{X_1}}
^0(-\log(E)))=0,$$ and this is the same as requiring $H^2(X_1,T_{X_1}^0)=0$ by \cite[Theorem 2]{LP07}. Let $X_2$ be the surface resulting of contracting the extremal P-resolution in $X_1$, and then partially resolving with the second extremal P-resolution of $(Q \in Y)$. Hence the surface $Y$ is the contraction of a smooth rational curve $C_2$ in $X_2$. So far, we have that $Y$ lives in a different moduli space than $X_1$ and $X_2$, but it is not clear whether $X_1$ and $X_2$ (or its KSBA replacement) live in the same moduli space or not.

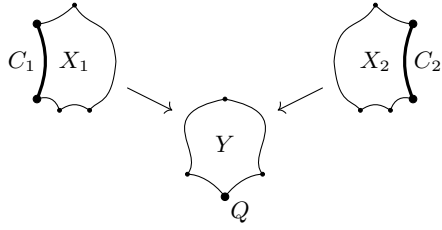
\begin{figure}[htbp]
\centering
\begin{tikzpicture}
\begin{scope}
\draw (-0.5,-0.5) .. controls (-0.35,-0.5) and (-0.15,-0.6) .. (0,-0.8);
\draw (0.5,-0.5) .. controls (0.35,-0.5) and (0.15,-0.6) .. (0,-0.8);
\draw (-0.5,-0.5) .. controls (-0.25,-0.2) and (-1,0.3) .. (0,0.5);
\draw (0.5,-0.5) .. controls (0.25,-0.2) and (1,0.3) .. (0,0.5);
\filldraw (-0.5,-0.5) circle [radius=0.025];
\filldraw (0.5,-0.5) circle [radius=0.025];
\filldraw (0,0.5) circle [radius=0.025];
\filldraw (0,-0.8) circle [radius=0.05];
\end{scope}

\begin{scope}[xshift=-2cm, yshift=1cm]
\draw (-0.2,-0.65) .. controls (-0.05,-0.55) and (0.05,-0.55) .. (0.2,-0.65);
\draw[very thick] (-0.5,-0.5) .. controls (-0.35,-0.15) and (-0.35,0.15) .. (-0.5,0.5);
\draw (0.2,-0.65) .. controls (0.3,-0.5) and (0.5,-0.6).. (0.55,0) .. controls (0.5,0.6) and (0.2,0.4) .. (0,0.75);
\draw (-0.5,-0.5) .. controls (-0.35,-0.45) and (-0.25,-0.5) .. (-0.2,-0.65);
\draw (-0.5,0.5) .. controls (-0.35,0.5) and (-0.15,0.6)..(0,0.75);
\filldraw (-0.5,-0.5) circle [radius=0.05];
\filldraw (-0.5,0.5) circle [radius=0.05];
\filldraw (0,0.75) circle [radius=0.025];
\filldraw (-0.2,-0.65) circle [radius=0.025];
\filldraw (0.2,-0.65) circle [radius=0.025];
\end{scope}

\begin{scope}[xshift=2cm, yshift=1cm, xscale=-1]
\draw (-0.2,-0.65) .. controls (-0.05,-0.55) and (0.05,-0.55) .. (0.2,-0.65);
\draw[very thick] (-0.5,-0.5) .. controls (-0.35,-0.15) and (-0.35,0.15) .. (-0.5,0.5);
\draw (0.2,-0.65) .. controls (0.3,-0.5) and (0.5,-0.6).. (0.55,0) .. controls (0.5,0.6) and (0.2,0.4) .. (0,0.75);
\draw (-0.5,-0.5) .. controls (-0.35,-0.45) and (-0.25,-0.5) .. (-0.2,-0.65);
\draw (-0.5,0.5) .. controls (-0.35,0.5) and (-0.15,0.6)..(0,0.75);
\filldraw (-0.5,-0.5) circle [radius=0.05];
\filldraw (-0.5,0.5) circle [radius=0.05];
\filldraw (0,0.75) circle [radius=0.025];
\filldraw (-0.2,-0.65) circle [radius=0.025];
\filldraw (0.2,-0.65) circle [radius=0.025];
\end{scope}

\draw[->] (-1.3,0.65) -- (-0.7,0.35);
\draw[->] (1.3,0.65) -- (0.7,0.35);

\begin{scriptsize}
\node at (0,-0.1) {$Y$};
\node at (-2,1) {$X_1$};
\node at (2,1) {$X_2$};
\node at (-2.7,1) {$C_1$};
\node at (2.7,1) {$C_2$};
\node at (0.2,-1) {$Q$};
\end{scriptsize}
\end{tikzpicture}
\caption{The three singular surfaces in a wormhole.} \label{whpic}
\end{figure}

\begin{lemma}
We have that $X_2$ defines a W-surface, and $K_{X_1}^2=K_{X_2}^2$ and $\chi(\O_{X_1})=\chi(\O_{X_2})$.
\end{lemma}

\begin{proof}
We need to prove existence of a $\Q$-Gorenstein smoothing for $X_2$. We know that $H^2(\widetilde{X_1},T_{\widetilde{X_1}}^0(-\log(E+\widetilde{C_1})))=0$. Let $A_1$ be the chain formed by the exceptional curves of the extremal P-resolution and $\widetilde{C_1}$. Let $A_2$ be the chain formed by the exceptional curves of the the second extremal P-resolution together with the corresponding curve $\widetilde{C_2}$. We know that to obtain $A_2$ we perform blow-downs until reaching the exceptional chain of $(Q \in Y)$, and then we perform blow-ups at that chain to obtain $A_2$. (We may not need blow-downs and/or blow-ups of course.) By the addition/deletion principle of $(-1)$-curves (see e.g. \cite[Prop.6]{LP07}) applied at each blow-down and blow-up , we have that $$H^2(\widetilde{X_1},T_{\widetilde{X_1}}^0(-\log(E+\widetilde{C_1})))=H^2(\widetilde{X_2},T_{\widetilde{X_2}}^0(-\log(E'+\widetilde{C_2})))$$ where $E'$ is the exceptional divisor of the minimal resolution $\widetilde{X_2} \to X_2$. Therefore, by our hypothesis, we have $H^2(\widetilde{X_2},T_{\widetilde{X_2}}^0(-\log(E'+\widetilde{C_2})))=0$. By using the standard short exact sequence $$0 \to T_{\widetilde{X_2}}^0(-\log(E'+\widetilde{C_2})) \to T_{\widetilde{X_2}}^0(-\log(E'))\to \mathcal{N}_{\widetilde{C_2}/ \widetilde{X_2}}\to 0,$$ we have that $H^2(\widetilde{X_2},T_{\widetilde{X_2}}^0(-\log(E')))=0$. Hence, by \cite[Theorem 2]{LP07}, we have that there are no-local-to-global obstructions to deform $X_2$, and so we have a W-surface $X_2$.   

In relation to invariants, since Wahl singularities are rational, we clearly have $\chi(\O_{X_1})=\chi(\O_{X_2})$. As for $K^2$, we note that if $X$ is a normal projective surface with only Wahl singularities and $\widetilde{X} \to X$ is the minimal resolution, then $K_X^2=K_{\widetilde{X}}^2+l$ where $l$ is the amount of exceptional curves. As $A_2$ is obtained by blow-downs and blow-ups on $A_1$ and we contract all curves except one, we obtain that $K_{X_1}^2=K_{X_2}^2$.  
\end{proof}

Therefore we have a W-surface $X_2$ with same invariants. However $K_{X_2}$ may not be nef. 

\begin{conjecture} [Wormhole conjecture] The MMP on the new W-surface finishes in a minimal model and it requires only flips, this is, both punctured W-surfaces live in the same moduli space. \label{wormhole}
\end{conjecture}

A main purpose of this paper is to show that Conjecture \ref{wormhole} is true when $X_1$ is not rational and for a wide range of WW singularities. One may hope that perhaps in the case when $X_1$ not rational we do have that $K_{X_2}$ is nef. We will prove that true in many situations, but the following example shows that it is not always the case.  

\begin{example} \label{ejenriques}
We consider an Enriques surface with the configuration of $(-2)$-curves shown in Figure \ref{enriques1}. This configuration is proved to exist in \cite[2.2]{DRU20}. 

\begin{figure}[htbp]
    \centering
    \begin{tikzpicture}
        \draw (0.0,-0.15) -- (0.0,1.35);
        \draw (-0.15,0.0) -- (0.2,0.0);
        \draw (0.2,0.0) arc (270:360:0.1);
        \draw (0.3,0.1) -- (0.3,1.1);
        \draw (0.3,1.1) arc (0:90:0.1);
        \draw (0.2,1.2) -- (-0.15,1.2);
        \draw (0.15,0.6) -- (1.95,0.6);
        \draw (0.6,0.45) -- (0.6,0.8);
        \draw (0.7,0.9) arc (90:180:0.1);
        \draw (0.7,0.9) -- (1.7,0.9);
        \draw (1.8,0.8) arc (0:90:0.1);
        \draw (1.8,0.8) -- (1.8,0.45);
        
        \begin{tiny}
        \draw (1.2,1.1) node {$A_1$};
        \draw (1.2,0.3) node {$A_2$};
        \draw (0.55,0.0) node {$A_3$};
        \draw (-0.25,0.6) node {$A_4$};
        \end{tiny}
    \end{tikzpicture}
    \caption{Special curves in an Enriques surface}
    \label{enriques1}
\end{figure}
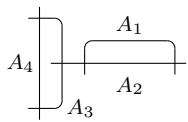

We do five blow-ups to get the configuration in Figure \ref{enriques2}. The exceptional curves $E_1, \dots, E_5$ are indexed according to the order of the blow-ups.

\begin{figure}[htbp]
    \centering
    \begin{tikzpicture}
        \draw (-0.1,0.7) -- (0.7,-0.1);
        \draw (0.5,-0.1) -- (1.3,0.7);
        \draw (1.1,0.7) -- (1.9,-0.1);
        \draw (1.7,-0.1) -- (2.5,0.7);
        \draw (2.3,0.7) -- (3.1,-0.1);
        \draw (2.85,0) -- (3.75,0);
        \draw (3.5,-0.1) -- (4.3,0.7);
        \draw (2.55,0.3) -- (4.05,0.3);
        \draw (0.0,0.45) -- (0.0,1.1);
        \draw (0.1,1.2) arc (90:180:0.1);
        \draw (0.1,1.2) -- (2.0,1.2);
        \draw (2.1,1.1) arc (0:90:0.1);
        \draw (2.1,1.1) -- (2.1,0.15);
        
        \begin{tiny}
        \draw (0.2,0.1) node {$E_4$};
        \draw (1.0,0.1) node {$E_3$};
        \draw (1.4,0.1) node {$A_4$};
        \draw (2.2,0) node {$A_3$};
        \draw (2.6,0.05) node {$A_2$};
        \draw (3.3,-0.2) node {$E_1$};
        \draw (3.3,0.5) node {$E_2$};
        \draw (4.0,0.05) node {$A_1$};
        \draw (1.05,1.4) node {$E_5$};
        \end{tiny}
    \end{tikzpicture}
    \caption{After five blow-ups.}
    \label{enriques2}
\end{figure}
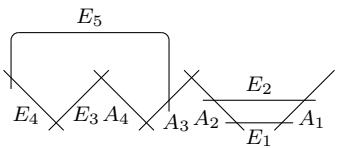

We have $E_1^2=E_2^2=E_5^2=-1, E_3^2=E_4^2=-2, A_4^2=-3, A_1^2=A_2^2=-4, A_3^2=-5$.The chain of curves 
\[ E_4-E_3-A_4-A_3-A_2-E_1-A_1, \] after contracting $E_1$ corresponds to the minimal resolution of the singularity $\frac{1}{235}(1, 169)$, since $[2, 2, 3, 5, 3, 3] = \frac{235}{169}$. This is a WW singularity, which define surfaces $X_1$ and $X_2$. In both cases we have W-surfaces $X_1$ and $X_2$ because we can prove they do not have obstructions (see \cite[Lemma 2.4]{DRU20}).

If we contract $A_1$ and $A_2-A_3-A_4-E_3-E_4$ to singularties $P_1$ and $P_2$, then we obtain the surface $X_1$ with the extremal P-resolution \[ [2, 2, 3, 5, 4]-1-[4]. \] It can be proved that a general $X_1$ has $K_{X_1}$ ample.

If we contract $E_3-A_4-A_3-A_2-A_1-E_1$ to a point $P_3$, then we get the surface $X_2$ with the extremal P-resolution \[ 2-[2, 3, 5, 3, 3]. \] But in this case we have $K \cdot E_5 = -\frac{1}{13}$. The curve $E_5$ induces a mk1A neighborhood. The numerical data for this mk1A is \[ [2, 3, 5-1, 3, 3]=\frac{129}{79}, \] which is not a Wahl singularity, and so this is a flipping mk1A. The extremal P-resolution which does the flip is \[ [2, 3, 5, 3] - 1 - [2, 5, 3].\] This is in Figure \ref{enriques3}, where $F_1, F_2, F_3$ are the new curves from the new blow-ups. We note that $E_2^2=E_4^2=-1$ and these are the only curves that could be negative for the canonical divisor $K$ of the new singular surface.  

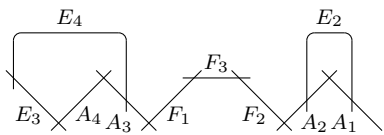
\begin{figure}[htbp]
    \centering
    \begin{tikzpicture}
        \draw (-0.1,0.7) -- (0.7,-0.1);
        \draw (0.5,-0.1) -- (1.3,0.7);
        \draw (1.1,0.7) -- (1.9,-0.1);
        \draw (1.7,-0.1) -- (2.5,0.7);
        \draw (2.25,0.6) -- (3.15,0.6);
        \draw (2.9,0.7) -- (3.7,-0.1);
        \draw (3.5,-0.1) -- (4.3,0.7);
        \draw (4.1,0.7) -- (4.9,-0.1);
        \draw (0.0,0.45) -- (0.0,1.1);
        \draw (0.1,1.2) arc (90:180:0.1);
        \draw (0.1,1.2) -- (1.4,1.2);
        \draw (1.5,1.1) arc (0:90:0.1);
        \draw (1.5,1.1) -- (1.5,0.15);
        \draw (3.9,0.15) -- (3.9,1.1);
        \draw (4.0,1.2) arc (90:180:0.1);
        \draw (4.0,1.2) -- (4.4,1.2);
        \draw (4.5,1.1) arc (0:90:0.1);
        \draw (4.5,1.1) -- (4.5,0.15);
        
        \begin{tiny}
        \draw (0.2,0.1) node {$E_3$};
        \draw (1.0,0.1) node {$A_4$};
        \draw (1.4,0) node {$A_3$};
        \draw (2.2,0.1) node {$F_1$};
        \draw (2.7,0.8) node {$F_3$};
        \draw (3.2,0.1) node {$F_2$};
        \draw (4.0,0) node {$A_2$};
        \draw (4.4,0) node {$A_1$};
        \draw (0.75,1.4) node {$E_4$};
        \draw (4.2,1.4) node {$E_2$};
        \end{tiny}
    \end{tikzpicture}
    \caption{After the flip.}
    \label{enriques3}
\end{figure}

However we compute $K \cdot E_4= \frac{1}{4}, K \cdot E_2=\frac{3}{5}$. In this way $K$ is now nef. We only used one flip to obtain the nef model, and we have a wormhole in the moduli space of $\Z/2$-Godeaux surfaces.
\end{example}

We now prove a relevant reduction step towards Conjecture \ref{wormhole}. Let us consider W-surfaces $X_1$ and $X_2$ as in Conjecture \ref{wormhole}. Let $(X_1 \subset \X'_1) \to (0 \in \D)$ be a partial $\Q$-Gorenstein deformation which keeps the distinguished extremal P-resolution in all fibers but smooths all other Wahl singularities. This is possible since we have $H^2(\widetilde{X_1},T_{\widetilde{X_1}}
^0(-\log(E+\widetilde{C_1})))=0$. We denote the general fiber by $X'_1$. Let $(X_2 \subset \X'_2) \to (0 \in \D)$ be the $\Q$-Gorenstein deformation obtained by first contracting the extremal P-resolution of all fibers in $(X'_1 \subset \X'_1) \to (0 \in \D)$ (where this deformation is trivial), and then gluing the other extremal P-resolution. The general fiber is denoted by $X'_2$. Since we do not have local-to-global obstructions, there are W-surfaces $X'_1$ and $X'_2$ as in the set-up of Conjecture \ref{wormhole}.

\begin{lemma}
If Conjecture \ref{wormhole} is true for the W-surfaces $X'_1, X'_2$, then it is also true for the W-surfaces $X_1, X_2$.
\label{reduction}
\end{lemma}

\begin{proof}
The point is that the $\Q$-Gorenstein deformation space of $X_i$'s and $X'_i$'s is smooth (see \cite[Section 3]{H12}). We have that the W-surface $X'_2$ has minimal model and requires only flips to obtain the KSBA replacement. Then the W-surface $X_2$ must satisfy the same since its $\Q$-Gorenstein deformation space is smooth and contains the one of $X'_2$.
\end{proof}

All in all, to verify that Conjecture \ref{wormhole} is true, we only need to verify it for W-surfaces $X_1$ which contain an extremal P-resolution over a WW singularity, so that it contains no other Wahl singularities out of this extremal P-resolution. That is the importance of Theorem \ref{W1W} and Theorem \ref{2W}, which will be proved in the next two sections. 

\section{Proof of Theorem \ref{W1W}} \label{s4}

In this section we essentially prove that the wormhole conjecture is valid for non rational surfaces with nef canonical class, and with an extremal P-resolution whose middle curve becomes a $(-1)$-curve in the minimal resolution. So the only possible counterexamples might come from extremal P-resolutions where the proper transform of the exceptional curve becomes a $(-m)$-curve with $m\geq 2$. At first they seem to be too special over a wormhole singularity, but they turn out to be chaotic. In the next section we manage to prove it only for $m=2$ in a special situation.   

Throughout this section we assume the hypothesis of Theorem \ref{W1W}, which we now recall. Let $Y$ be a normal projective surface with one cyclic quotient singularity $(Q \in Y)$, which is smooth everywhere else. We assume that the minimal resolution of $Y$ is not ruled, and that $Q$ is a wormhole singularity, i.e. it admits two extremal P-resolutions $f_i^+ \colon (C_i \subseteq X_i) \to (Q \in Y), i=1, 2$. In addition we assume:

\begin{itemize}
\item The strict transform in the minimal resolution of $X_2$ of the exceptional curve $C_2$ for the extremal P-resolution in $X_2$ is a $\P^1$ with self-intersection $-1$. 
\item The canonical class $K_{X_1}$ is nef.
\item Both surfaces $X_i$ admit $\Q$-Gorenstein smoothings $(X_i \subseteq \X_i) \to (0 \in \D)$, i.e. they are W-surfaces.
\end{itemize}

We want to prove that $K_{X_2}$ is nef. This implies that the family $(X_i \subseteq \X_i) \to (0 \in \D)$ has nef canonical class (see e.g. \cite[Sect.2]{U16a}).
\bigskip

Let $\frac{1}{\Delta}(1, \Omega)=(Q \in Y)$, and $\frac{\Delta}{\Omega} = [f_s, \dots, f_1] - 1 - [e_1, \dots, e_r] $ be the numerical data of the extremal P-resolution $X_2 \to Y$. Let $\sigma \colon \tilde{X_2} \to X_2$ be the minimal resolution of $X_2$ over the singularities $P_1$ and $P_2$. Let $E_i,F_j$ be the $\P^1$'s which resolve them respectively. In this way we have $E_i^2=-e_i$ and $F_j^2=-f_j$.

Let us assume that $K_{X_2}$ is not nef. By hypothesis we have existence of $(X_2 \subseteq \X_2) \to (0 \in \D)$, and so we know that there is a curve $\Gamma \simeq \P^1$ in $X_2$ such that $K_{X_2} \cdot \Gamma <0$ (see e.g. \cite[Sect. 2]{U16a}). Since $X_2$ is not ruled, we can assume that $\Gamma^2<0$, and $(\Gamma \subset X_2 \subseteq \X_2) \to (Q \in Y \subset \mathcal{Y})$ is an extremal neighborhood of type mk1A or mk2A. In this way, the curve $\Gamma$ has a very special position in relation to the singularities of $X_2$. Also the assumption that $K_{X_1}$ is nef puts more constraints, which can be summarized as:

\begin{itemize}
\item Necessarily $\Gamma$ intersects $(f_1^+)^{-1}(Q)$, since otherwise $\Gamma$ would be negative for $K_{X_1}$.

\item The curve $\Gamma$ cannot intersect $C_2$ out of the singularities $P_1, P_2$, since otherwise we can contract $\Gamma$ in $\tilde{X_2}$ producing a surface $X'$ and a curve $C_2$ with $K_{X'} \cdot C_2 <-1$. But this is contrary to our assumption that $\tilde{X_2}$ is not ruled (and so it has a minimal model).

\item As we have an mk1A or mk2A situation, the curve $\Gamma$ in $\tilde{X_2}$ can touch one Wahl chain transversally at one point, or both chains transversally at the ends of each. The first option is not possible since either it becomes a negative curve for $K_{X_1}$ or we have contradiction with the not ruled assumption.
\end{itemize}

Therefore the curve $\Gamma$ can only intersect the $F_1,F_s$ and the $E_1, E_r$ in a mk2A situation (four possibilities). In the next arguments, we will strongly use the discrepancies of the two Wahl singularities. We recall that $$K_{X_2} \cdot \Gamma = \Big( K_{\tilde{X_2}} - \sum_a k_a E_a - \sum_b l_b F_b \Big) \cdot \Gamma=-1-k_i-l_j $$ where $k_a,l_b$ are the discrepancies of the corresponding divisors, and $i=1,r$ and $j=1,s$ are the only possibilities. We can easily discard two of the four possibilities:

\begin{itemize}
\item If $\Gamma$ intersects $E_1$ and $F_1$, then $K_{X_1} \cdot \Gamma = K_{X_1} \cdot C_2>0$ because both curves become $(-1)$-curves in the minimal resolution. 

\item If $\Gamma$ intersects $E_r$ and $F_s$, then the extremal P-resolution on $X_1$ must have two singularities (since otherwise $\Gamma$ will intersect only once the singularity, and so it will be negative for $K_{X_1}$). In this way, and as in the proof of Theorem \ref{calcularc}, it follows that $\Gamma$ must intersect the extreme curves of the two chains from the minimal resolution of $X_1$. By the same result, we know that in this case the strict transform of $C_1$ in the minimal resolution of $X_1$ is a $(-1)$-curve. We also know that discrepancies at the end of a Wahl chain add up $-1$. Therefore we obtain $$ K_{X_1} \cdot \Gamma + K_{X_1} \cdot C_1=0 $$ but $K_{X_1} \cdot C_1>0$, and so a contradiction.
\end{itemize}

The third and fourth possibilities are symmetric, so without loss of generality we assume that $\Gamma$ is intersecting $E_1$ and $F_s$ as in Figure \ref{curvagammaenres}.

\begin{figure}[htbp]
    \centering
    \begin{tikzpicture}
        \draw (-0.1,0.7) -- (0.7,-0.1);
        \draw[dotted] (0.5,-0.1) -- (1.3,0.7);
        \draw (1.1,0.7) -- (1.9,-0.1);
        \draw (1.65,0) -- (2.55,0);
        \draw (2.3,-0.1) -- (3.1,0.7);
        \draw[dotted] (2.9,0.7) -- (3.7,-0.1);
        \draw (3.5,-0.1) -- (4.3,0.7);
        \draw (0,0.45) -- (0,1.1);
        \draw (0.1,1.2) arc (90:180:0.1);
        \draw (0.1,1.2) -- (2.6,1.2);
        \draw (2.7,1.1) arc (0:90:0.1);
        \draw (2.7,1.1) -- (2.7,0.15);
        
        \begin{scriptsize}
        \draw (0.2,0) node {$F_s$};
        \draw (1.4,0) node {$F_1$};
        \draw (2.1,-0.3) node {$C_2$};
        \draw (2.8,-0.1) node {$E_1$};
        \draw (4.0,0) node {$E_r$};
        \draw (1.35,1.35) node {$\Gamma$};
        \end{scriptsize}
    \end{tikzpicture}
    \caption{The potential bad curve $\Gamma$ in $\tilde{X_2}$.}
    \label{curvagammaenres}
\end{figure}
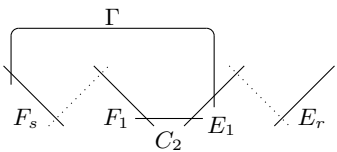

We note that we must have $r>1$, since $r=1$ would give a $\Gamma$ intersecting $E_r=E_1$ and $F_s$, but this case was ruled out above.

\begin{proposition} \label{propcompartida}
Let $Z$ be a normal projective surface, $P_1, P_2 \in Z$ the only singular points, which are Wahl singularities, and let $\sigma \colon \tilde{Z} \to Z$ be the minimal resolution of $Z$, which is not ruled. Assume that there exists $(-1)$-curves $C$ and $\Gamma$, such that on the minimal resolution we have the configuration given by Figure \ref{curvagammaenres} (taking $C=C_2$), where $E_1, \dots, E_r$ and $F_1, \dots, F_s$ are the resolutions of $P_1$ and $P_2$. Assume also that $r>1$. Then, we cannot have simultaneously $K_Z \cdot C>0$ and $K_Z \cdot \Gamma <0$. 
\end{proposition}

This proposition allows us to finish the proof of Theorem \ref{W1W}. It will be used also in the next section.

\begin{proof} (of Theorem \ref{W1W})
Assume that $K_{X_2}$ is not nef. As we discussed above, we get a rational curve $\Gamma$, which is negative on $K_{X_2}$ and positive on $K_{X_1}$, and which gives us Figure \ref{curvagammaenres} on the minimal resolution (with $r>1$). We then have that $K_{X_2}\cdot C_2>0$ and $K_{X_2} \cdot \Gamma<0$, which contradicts Proposition 4.1 with $Z=X_2$ and $C=C_2$.
\end{proof}

The proof of Proposition \ref{propcompartida} will be achieved by means of the next few lemmas.

\begin{lemma}
We must have $s>1$. 
\end{lemma}

\begin{proof}
If $s=1$, then $K_{Z} \cdot \Gamma = K_{Z} \cdot C>0$.
\end{proof}

\begin{lemma}
We must have $e_r=f_s=2$.
\end{lemma}

\begin{proof}
As $r,s>1$, have have that exactly one of the values $e_1, e_r$ is $2$, the same for $f_1, f_s$. We will verify that the other $3$ cases for $e_r,f_s$ are impossible:

\begin{itemize}
\item If $e_1=f_1=2$, then contracting in the configuration $E_1, C, F_1$ we obtain a $\mathbb{P}^1$ with self-intersection equal to $0$. But this is a contradiction with the not ruled assumption on $\tilde{Z}$.

\item If $e_1=f_s=2$, then the argument against is analogue to the previous with $E_1, \Gamma, F_s$.

\item Let $e_r=f_1=2$. We have $K_{Z} \cdot \Gamma = -1-k_1-l_s$. When we compute the values of $\delta_i$ for $[e_1, \dots, e_r]$, we obtain that $\delta_1<\delta_r$, and so $$ k_1 = -1 + \frac{\delta_1}{\delta_1+\delta_r} < -1 + \frac{\delta_1}{\delta_1+\delta_1} = -\frac{1}{2}.$$ An analogue argument shows that $l_s < -\frac{1}{2}$, and so  $$  K_{Z} \cdot \Gamma > -1 + \frac{1}{2}+\frac{1}{2}=0.$$
\end{itemize}
This shows that the only option is $e_r=f_s=2$.
\end{proof}

The previous argument used a very simple observation on discrepancies of Wahl singularities. To continue the proof of Proposition \ref{propcompartida} we need a more general statement on these discrepancies.

\begin{lemma}\label{cotadiscrepancias}
Let $[b_1, \dots, b_t]$ be a Wahl singularity, assume $t \geq 2$ and $b_t=2$, and let us denote its discrepancies by $m_1, \dots, m_t$. Then we have the following bounds:
\begin{itemize}
\item[(Type $M$)] If $b_2=b_3=\dots=b_t$, then $ m_1 = -1+\frac{1}{b_1-2}$ and $m_t = -\frac{1}{b_1-2}$.
\item[(Type $B$)] Otherwise $m_1 = -1+\mu$ and $m_t = -\mu$,
where $\frac{1}{b_1} < \mu < \frac{1}{b_1-1}$. 
\end{itemize}
\end{lemma}

\begin{proof}
We will use again the $\delta_i$ as in Equation \ref{ecdiscwahl}.

(Type $M$): Every such singularity comes from $[4]$ adding $2$'s to the right. In this way $\delta_1=1, \ \delta_2=2, \ \dots, \ \delta_t=t$. Then the discrepancies are $m_1 = -1 + \frac{1}{t+1}$ and $m_t = -1 + \frac{t}{t+1}$. As $b_1=t+3$, we get what we wanted.

(Type $B$): Let us say that $b_1=p+2$. Eliminating the $2$'s on the right, we fall into $[2, b_2, \dots, b_{t-p}]$, with $b_{t-p}>2$. In this way, $\delta_1 > \delta_{t-p}$, because the first entry is a $2$. Adding back the $2$'s on the right, we get
\[ \delta_{t-p+i} = \delta_{t-p}+i\delta_1, \qquad i = 0, \dots, p. \]
In particular, $\delta_t = \delta_{t-p}+p\delta_1$. Hence, if $\mu = \frac{\delta_1}{\delta_1+\delta_t}$, then 
\[ m_1 = -1 + \frac{\delta_1}{\delta_1+\delta_t}=-1+\mu \]
and
\[ m_t = -1 + \frac{\delta_t}{\delta_1+\delta_t}=-\frac{\delta_1}{\delta_1+\delta_t} = -\mu. \]
It is enough then to bound $\mu$. As $\delta_1>\delta_{t-p}$, we have
\[ \mu > \frac{\delta_1}{\delta_1+(\delta_1+p\delta_1)}=\frac{1}{p+2}=\frac{1}{b_1}. \]
On the other hand, as all $\delta_i$ are positive,
\[ \mu = \frac{1}{p+1} \cdot \frac{(p+1)\delta_1}{(p+1)\delta_1+\delta_{t-p}}< \frac{1}{p+1}=\frac{1}{b_1-1}. \]
These two bounds give $\frac{1}{b_1}<\mu <\frac{1}{b_1-1}$.
\end{proof}

We now continue the proof of Proposition \ref{propcompartida}.

\begin{lemma} \label{notipoM}
Necessarily $[e_1, \dots, e_r]$ must be of type $M$. Moreover, if $[f_1, \dots, f_s]$ is of type $B$, then $e_1=f_1+1$; If $[f_1, \dots, f_s]$ is of type $M$, then $e_1=f_1-1$.
\end{lemma}

\begin{proof} 
The basic idea is to see what happens to $E_1$ in $\tilde{Z}$ after we contract all possible $(-1)$-curves. We can contract $\Gamma$, $C$, and then all $(-2)$-curves at the end of the $F_j$ chain $F_s, \dots, F_{s'}$. This will impose conditions to $e_1$ and $f_1$, which will allow to bound the discrepancies involved in $K_Z \cdot \Gamma$.

We are going to analyze the four possible case, which depend on the type $B$ or $M$ of the singularities.

\bigskip
$\textbf{(BB):}$ If $[e_1, \dots, e_r]$ and $[f_1, \dots, f_s]$ are of type $B$, then we have $f_1-2$ entries $2$ starting with $f_s$, and so the curve $E_1$ will have self-intersection $-e_1+1+1+(f_1-2)=-e_1+f_1$ after we contract $\Gamma, C$ and $\{F_s, \dots, F_{s'}\}$. Because of our not ruled assumption on $\tilde{Z}$, we must have $e_1 \geq f_1+1$. By Lemma \ref{cotadiscrepancias}, we have 
\[ k_1 < -1+\frac{1}{e_1-1}, \quad l_s < -\frac{1}{f_1}. \]
Therefore,
\[ K_{Z} \cdot \Gamma > -1 + 1-\frac{1}{e_1-1} + \frac{1}{f_1}= \frac{e_1-1-f_1}{(e_1-1)f_1} \geq 0, \]
since $e_1 \geq f_1+1$.

\bigskip
$\textbf{(MB):}$ If $[f_1, \dots, f_s]$ is of type $M$ and $[e_1, \dots, e_r]$ is of type $B$, then we have $s-1=f_1-4$ entries $2$ starting with $f_s$, and so the curve $E_1$ will have self-intersection $-e_1+1+1+(f_1-4)=-e_1+f_1-2$ after we contract $\Gamma, C$ and $\{F_s, \dots, F_{s'}\}$. Again, because of our not ruled assumption on $\tilde{Z}$, we must have $e_1 \geq f_1-1$. The bound for $k_1$ is as above, while $l_s = -\frac{1}{f_1-2}$. In this way,
\[ K_{Z} \cdot \Gamma > -1 + 1-\frac{1}{e_1-1}+\frac{1}{f_1-2}= \frac{e_1-f_1+1}{(e_1-1)(f_1-2)} \geq 0, \]
since $e_1 \geq f_1-1$.

\bigskip
$\textbf{(BM):}$ If $[f_1, \dots, f_s]$ is of type $B$ and $[e_1, \dots, e_r]$ is of type $M$, then we have $f_1-2$ entries $2$ starting with $f_s$, and so the curve $E_1$ will have self-intersection $-e_1+f_1$ after the contractions as above, and so $e_1 \geq f_1+1$. By Lemma \ref{cotadiscrepancias} we can write
\[ k_1 = -1 + \frac{1}{e_1-2}, \quad l_s < -\frac{1}{f_1}, \]
and so
\[ K_{Z} \cdot \Gamma > -1 + 1 - \frac{1}{e_1-2}+\frac{1}{f_1} = \frac{e_1-2-f_1}{(e_1-2)f_1}. \]
If $e_1 \geq f_1+2$, then $K_{Z} \cdot \Gamma >0$, and so we necessarily get $e_1=f_1+1$.

\bigskip
$\textbf{(MM):}$ If $[e_1, \dots, e_r]$ and $[f_1, \dots, f_s]$ are of type $M$, we have $f_1-4$ entries $2$ starting with $f_s$, and so the curve $E_1$ will have self-intersection $-e_1+f_1-2$ after the contractions as above, and so $e_1 \geq f_1-1$. By Lemma \ref{cotadiscrepancias} we have
\[ k_1 = -1 + \frac{1}{e_1-2}, \quad l_1 = -\frac{1}{f_1-2}, \]
and so 
\[ K_{Z} \cdot \Gamma = -1 + 1 - \frac{1}{e_1-2} + \frac{1}{f_1-2} = \frac{e_1-f_1}{(e_1-2)(f_1-2)}. \]
If $e_1 \geq f_1$, then $K_{Z} \cdot \Gamma >0$, and so we necessarily get $e_1=f_1-1$. 
\end{proof}

\begin{lemma} \label{solotipoMB}
Necessarily $[e_1, \dots, e_r]$ is of type $M$, and $[f_1, \dots, f_s]$ is of type $B$.
\end{lemma}

\begin{proof}
By Lemma \ref{notipoM}, the only other possibility is that $[e_1, \dots, e_r]$ and $[f_1, \dots, f_s]$ are of type $M$, together with $e_1=f_1-1$. Let $q:=f_1 \geq 5$, we have
\[ \frac{\Delta}{\Omega} = [\underbrace{2, \dots, 2}_{q-4}, q]-1-[q-1, \underbrace{2, \dots, 2}_{q-5}]. \]
We have $r=q-4, s=q-3$. In $\tilde{Z}$ we have a situation as in Figure \ref{casoMMuno}, where $C^2=\Gamma^2=-1, F_1^2=-q, E_1^2=-(q-1)^2$, and all the rest are $(-2)$-curves.

\begin{figure}[htbp]
\centering
\begin{tikzpicture}
\draw (-0.1,0.7) -- (0.7,-0.1);
\draw (0.5,-0.1) -- (1.3,0.7);
\draw[dotted] (1.1,0.7) -- (1.9,-0.1);
\draw  (1.7,-0.1) -- (2.5,0.7);
\draw  (2.3,0.7) -- (3.1,-0.1);
\draw  (2.9,0.0) -- (3.7,0.0);
\draw  (3.5,-0.1) -- (4.3,0.7);
\draw  (4.1,0.7) -- (4.9,-0.1);
\draw  (4.7,-0.1) -- (5.5,0.7);
\draw[dotted] (5.3,0.7) -- (6.1,-0.1);
\draw (5.9,-0.1) -- (6.7,0.7);
\draw (0.0,0.45) -- (0.0,1.1);
\draw (0.1,1.2) arc (90:180:0.1);
\draw (0.1,1.2) -- (3.8,1.2);
\draw (3.8,1.2) arc (90:0:0.1);
\draw (3.9,1.1) -- (3.9,0.15);

\begin{scriptsize}
\draw (3.95,-0.1) node {$E_1$};
\draw (4.4,0) node {$E_2$};
\draw (5.2,0) node {$E_3$};
\draw (6.4,0) node {$E_r$};
\draw (2.6,0) node {$F_1$};
\draw (2.2,0) node {$F_2$};
\draw (1.1,0) node {$F_{s-1}$};
\draw (0.2,0) node {$F_s$};
\draw (3.3,-0.25) node {$C$};
\draw (1.95,1.35) node {$\Gamma$};
\end{scriptsize}

\end{tikzpicture}
\caption{Situation in $\tilde{Z}$, case $MM$.}
\label{casoMMuno}
\end{figure}
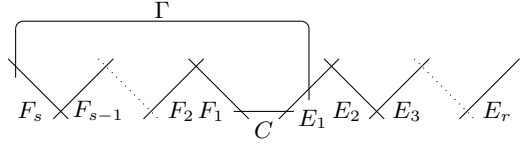

After we contract $\Gamma$, $C$, $F_{q-3}, \ldots, F_2$, we obtain that $F_1$ and $E_1$ form a cycle followed by the chain $E_2,\ldots,E_r$, where $E_1^2=-1, F_1^2=-(q-2)$, y all the rest are $(-2)$-curves. We can now contract $E_1$, so $F_1$ becomes a nodal rational curve and $F_1^2=-(q-6)$. We now keep contracting $E_2,\ldots,E_r$ obtaining that \[ K \cdot F_1 = (q-6)-2 \cdot (r-1)=(q-6)-2(q-5)=4-q<0, \]
since $q \geq 5$. This gives a contradiction since $\tilde{Z}$ is not ruled. 
\end{proof}

With the previous lemma, the only option is $[e_1, \dots, e_r]$ of type $M$, and $[f_1, \dots, f_s]$ of type $B$. Hence we can write \[ [f_1, \dots, f_s]= [p, f_2, \dots, f_t, \underbrace{2, \dots, 2}_{p-2}], \] where $t \geq 2$. Cancelling the $2$s on the right, we get $[2, f_2, \dots, f_t]$ which is a Wahl chain again. As its length is at least two, we have that $[f_t, \dots, f_2, 2]$ is of type $M$ or $B$. In the next final lemmas, we will say that $[f_1, \dots, f_s]$ is of type $BM$ or $BB$ respectively.

\begin{lemma} \label{notipoMBM}
If $[e_1, \dots, e_r]$ is of type $M$, then $[f_1, \dots, f_s]$ must be of type $BB$. 
\end{lemma}

\begin{proof}
We assume that $[e_1, \dots, e_r]$ is of type $M$, and $[f_1, \dots, f_s]$ is of type $BM$. Let $p=f_1=e_1-1 \geq 4$. Then we have  \[ [e_1, \dots, e_r]=[p+1, \underbrace{2, \dots, 2}_{p-3}] \] and \[ [f_1, \dots, f_s]=[p, \underbrace{2, \dots, 2}_{q-5}, q, \underbrace{2, \dots, 2}_{p-2}], \]
where $q \geq 5$. Let $t=q-3$ be the position of the entry equal to $q$. In this way $E_1^2=-(p+1), F_1^2=-p, F_t^2=-q, C^2=\Gamma^2=-1$, and all other curves in this situation are $(-2)$-curves.

We will achieve a contradiction showing that $K \cdot F_t$ is eventually negative with $F_t$ singular, which goes against the assumption $\tilde{Z}$ is not ruled. We first contract $C$, $\Gamma$, $F_s, \ldots,F_{t+1}$. Then $E_1$ becomes a $(-1)$-curve. We then contract $E_1,\ldots,E_r$ and so $F_1$ becomes a $(-1)$-curve. If $q>5$, then $F_t$ intersects $F_1$ only at one point with $F_t \cdot F_1=p-2$. By contracting $F_1, \dots, F_{t-1}$ we get a singular curve $F_t$ with $K \cdot F_t = -(p-3)(q-3)-1<0$. If $q=5$, then $F_t$ intersects $F_1$ at two points, with $F_t \cdot F_1=1+(p-2)$. After contracting $F_1$, we get a singular $F_t$ with $K \cdot F_t=5-2p<0$.
\end{proof}

\begin{lemma} \label{notipoMBB}
The case $[e_1, \dots, e_r]$ of type $M$ and $[f_1, \dots, f_s]$ of type $BB$ is impossible.
\end{lemma}

\begin{proof}
Let $p=f_1=e_1-1 \geq 4$. We can write \[ [e_1, \dots, e_r]=[p+1, \underbrace{2, \dots, 2}_{p-3}] \] and \[ [f_1, \dots, f_s]=[p, \underbrace{2, \dots, 2}_{q-3}, \dots, q, \underbrace{2, \dots, 2}_{p-2}], \]
where $q \geq 3$. Let $t=s-(p-2)$ be the position of the entry $q$. The contractions that will come are exactly the contractions we perform in the previous lemma, but at the end we are contracting $F_1, \ldots, F_{q-2}$. The relevant intersection now is 
\[ K \cdot F_t = (q-p-1)-(q-2)(p-2)=-(p-3)(q-1)-2<0. \] 
\end{proof}

With Lemma \ref{notipoMBB} we finish the proof of Proposition \ref{propcompartida}, and so Theorem \ref{W1W}.

\section{Proof of Theorem \ref{2W}} \label{s5}

Throughout this section we assume the hypothesis of Theorem \ref{2W}, which we now recall. Let $Y$ be a normal projective surface with one cyclic quotient singularity $(Q \in Y)$, which is smooth everywhere else. We assume that the minimal resolution of $Y$ is not ruled, and that $Q$ is a wormhole singularity, i.e. it admits two extremal P-resolutions $f_i^+ \colon (C_i \subseteq X_i) \to (Q \in Y), i=1, 2$. In addition we assume:

\begin{itemize}
\item The strict transform in the minimal resolution of $X_2$ of the exceptional curve $C_2$ for the extremal P-resolution in $X_2$ is a $\P^1$ with self-intersection $-2$, and $X_2$ has only one singularity. 
\item The canonical class $K_{X_1}$ is nef.
\item Both surfaces $X_i$ admit $\Q$-Gorenstein smoothings $(X_i \subseteq \X_i) \to (0 \in \D)$, i.e. they are W-surfaces.
\end{itemize}

We want to prove that we only need flips to run MMP on $(X_2 \subset \X_2) \to (0 \in \D)$. Here we cannot guarantee that $K_{X_2}$ is nef, we indeed may need some flips, as shown by Example \ref{ejenriques}. The proof will be substantially different to the proof of Theorem \ref{W1W}, and lemmas will take a more general situation than the one we started with.

\begin{lemma} \label{controlcurva2Wahl}
Let $\tilde{Z}$ be a smooth projective surface which is not ruled. Suppose that $\tilde{Z}$ has some chain of smooth rational curves $C, E_1, \dots, E_r$, with $C^2=-2, E_i^2=-b_i$, $b_i \geq 2$, and $[b_1, \dots, b_r]=\frac{m^2}{ma-1}$ is a Wahl chain. Suppose also that we have a $(-1)$-curve $\Gamma$ which transversely intersects only one $E_j$ at one point, and also intersects $C$. Then, it follows that $\Gamma$ intersects $C$ transversely at one point, $b_j \neq 2$, and $j \neq r$.
\end{lemma}

This lemma will be useful when we have a mk1A neighborhood via $\Gamma$ over an extremal P-resolution with only one Wahl singularity and a $(-2)$-curve. We can take $\tilde{Z}$ as the minimal resolution of the singularity, and $E_1, \dots, E_r$ the exceptional divisor.

\begin{proof}
Note first that $K \cdot C=0$ by adjunction. If we blow-down $\Gamma$, then the intersection $K \cdot C$ decreases in $\Gamma \cdot C$. But canonical class must be eventually nef, and so the only possibility is  $\Gamma \cdot C=1$.

After blowing-down $\Gamma$, we can blow-down $C$, and so $E_j^2\leq -3$ as $\tilde{Z}$ is not ruled, i.e. $b_j \neq 2$.

Suppose now that $j=r$. Since $b_j>2$, we have $b_1=2$ (or $r=1$, where $b_1=b_j=4$, which leads to a straightforward contradiction). We have two options.
\begin{itemize}
\item If $[b_r, \dots, b_1]$ is a Wahl singularity of type $M$, so that $[b_1, \dots, b_r]=[2, \dots, 2, r+3]$, we can blow-down $\Gamma, C$, and $E_1, E_2, \dots, E_{r-1}$. We get a nodal curve $E_r$ with $K \cdot E_r=-1$, which is a contradiction, since $\tilde{Z}$ is not ruled.

\item If $[b_r, \dots, b_1]$ is a Wahl singularity of type $B$, so that $$ [b_1, \dots, b_r]=[2, \dots, 2, b_{s+1}, \dots, b_{r-1}, s+2], $$
we can blow-down $\Gamma, C$, and $E_1, \dots, E_s$. Thus, we get $K \cdot E_r=-2$, which is again a contradiction.
\end{itemize}
It follows that $j \neq r$. 
\end{proof}

The next lemma will be useful to control mk1A neigborhoods which will appear on the proof of Theorem \ref{2W}. Roughly speaking, when we run MMP to $(X_2 \subset \X_2) \to (0 \in \D)$, we can get an extremal P-resolution with either two singularities and a $(-1)$-curve in the middle, or just one singularity with a $(-2)$-curve.

\begin{lemma} \label{mk1Aesflip}
Assume that the same hypotheses of Lemma \ref{controlcurva2Wahl} hold. Let $Z$ be the surface obtained contracting $E_1, \dots, E_r$, and assume that $Z$ admits a $\Q$-Gorenstein smoothing $(Z \subset \mathcal{Z}) \to (0 \in \D)$. Then $\Gamma \subseteq Z$ induces a mk1A neighborhood which must be of flipping type. The resulting extremal P-resolution after the flip must have either two Wahl singularities with a $(-1)$-curve in the middle, or one Wahl singularity with a $(-2)$-curve. 
\end{lemma}

\begin{proof}
Let $\sigma\colon \tilde{Z} \to Z$ be the minimal resolution of $Z$. The curve $\Gamma$ induces a mk1A neighborhood on $Z$. Note that Lemma \ref{controlcurva2Wahl} says that $b_j$ must be $> 2$. Hence we have the mk1A neigborhood $(\Gamma \subset Z) \to (Q \in Y) =\frac{1}{\Delta}(1, \Omega)$ where
$$ \frac{\Delta}{\Omega} = [b_1, \dots, b_j-1, \dots, b_r] $$
has every entry $\geq 2$. Since $[b_1, \dots, b_r]$ is a Wahl singularity, $\sum b_i=3r+1$, and then the sum of the entries of $\frac{\Delta}{\Omega}$ is $3r$. This proves that $\frac{1}{\Delta}(1, \Omega)$ is not a Wahl singularity, and then we have a flipping mk1A (see Subsection 3.1). 

After we flip, we obtain a new W-surface $Z'$, together with a extremal P-resolution over $(Q \in Y)$. In this way, we can apply Theorem 2.5 to it, since we computed before the sum of the entries of the minimal resolution of $(Q \in Y$). We get that if the new extremal P-resolution has one singularity, then self-intersection of the flipping curve (on the new minimal resolution) must be $-2$; if there are two singularities, then this self-intersection must be $-1$ (on the new minimal resolution).
\end{proof}

The proof of Theorem \ref{2W} will be based on a repeated use of Lemma \ref{mk1Aesflip}. We will need to control the new outcomes from Lemma \ref{mk1Aesflip}. For that, we give a definition for these two cases.  

\begin{definition}
Let $Y$ be a normal projective surface with one cyclic quotient singularity $Q \in Y$. We name the following extremal P-resolutions $(C \subset Z) \to (Q \in Y)$ as follows:

\begin{itemize}
\item[\textbf{Type(-1)}:] The surface $Z$ has two singularities, and the strict transform of $C$ in the minimal resolution of $Z$ is a $(-1)$-curve.

\item[\textbf{Type(-2)}:] The surface $Z$ has one singularity, and the strict transform of $C$ in the minimal resolution of $Z$ is a $(-2)$-curve.
\end{itemize}
\label{-1-2}
\end{definition}

We are not assuming $K_Z$ is ample, it is only required $C \cdot K_Z >0$.

\begin{lemma} \label{-2}
Let us consider the hypothesis of Theorem \ref{2W}. Let $Z_1$ be the W-surface $X_2$. Assume we have run the MMP on W-surfaces $Z_1, \ldots, Z_m$ so that the flip from $(\Gamma_i \subset Z_i)$ to $(C_{i+1} \subset Z_{i+1})$ comes always from a Type(-2) extremal P-resolution as in Lemma \ref{mk1Aesflip}. In addition, assume that $K_{Z_m}$ is not nef. Then the only possible mk1A for $Z_m$ is the one described in Lemma \ref{controlcurva2Wahl}.
\end{lemma}

\begin{proof}
If not, we have a curve $\Gamma_{m} \subset Z_m$ so that it is a $(-1)$-curve in the minimal resolution $\tilde{Z}_m$, it is disjoint from $C_m$, and it intersects only $E_i$ transversally at one point. We note that the flipping curves $\Gamma_i \subset Z_i$ satisfy Lemma \ref{controlcurva2Wahl}. As none of the $Z_i$ are ruled, in particular $Z_m$, the curves $\Gamma_{m-1}, \Gamma_{m}$ are disjoint in $\tilde{Z}_m$. Hence $\Gamma_{m}$ is again a $(-1)$-curve in $\tilde{Z}_{m-1}$. Inductively we obtain a $(-1)$-curve $\Gamma_m$ in $\tilde{Z}_1$ which is disjoint from $C_1$ and only intersects some $E_j$ transversally at one point. We now go to $X_1$. Since $X_1$ is not ruled, the curve $\Gamma_m$ must be a $(-1)$-curve in $\tilde{X}_1$ intersecting only one exceptional curve of $\tilde{X}_1 \to X_1$ transversally at one point. But that is not possible since then $K_{X_1} \cdot \Gamma_m <0$. Therefore the curve $\Gamma_m$ must intersect $C_m$, and so we are as in Lemma \ref{controlcurva2Wahl}.
\end{proof}

\begin{lemma} \label{-1}
Let us consider the hypothesis of Theorem \ref{2W}. Let $Z_1$ be the W-surface $X_2$. Assume we have run the MMP on W-surfaces $Z_1, \ldots, Z_m$ so that the flip from $(\Gamma_i \subset Z_i)$ to $(C_{i+1}\subset Z_{i+1})$ comes always from a Type(-2) extremal P-resolution as in Lemma \ref{mk1Aesflip} for $i=1,\ldots,m-2$, and the last one is Type(-1). In addition, assume that $K_{Z_m}$ is not nef. Then it is not possible to have a mk1A neighborhood for $Z_m$.
\end{lemma}

\begin{proof}
The proof is similar to the proof of Lemma \ref{-2}. A potential $\Gamma_m \subset Z_m$ defining a mk1A will not intersect $C_m, C_{m-1}, \dots, C_1$. Hence it will survive untouched until reaching $Z_1$, giving a mk1A neighborhood to $Z_1$, and in particular it will be a negative curve, but we know that this is not possible since $K_{Z_1}$ is nef. It is key that the surfaces involved are not ruled, so that $(-1)$-curves remain disjoint.
\end{proof}

We now show a key step to rule out certain mk2A neighborhood. After that we will have everything to give a proof for Theorem \ref{2W}.

\begin{lemma} \label{prelemacompartido}
Let $Z$ be a normal projective surface, $Q_1, Q_2$ the only singular points on $Z$. Assume that there is a $(-1)$-curve $D$ passing through $Q_1$ and $Q_2$, such that $(D \subset Z) \to (Q \in Y)$ is an extremal P-resolution. 

Let $\sigma\colon \tilde{Z} \to Z$ be the minimal resolution of $Z$, which is not ruled, with $F_1, \dots, F_s$ and $G_1, \dots, G_t$ the exceptional divisors for $Q_1$ and $Q_2$. Suppose that we have two $(-1)$-curves $B$ and $\Gamma$ on $Z$, such that on $\tilde{Z}$ the configuration is as on Figure \ref{de2WaW1W}, and $B$ intersect transversally other curve in the set $\{F_i, G_j\}$ (different of $G_t$ and $F_s$) in a point. Then $\Gamma \cdot K_Z \geq 0$ .

\begin{figure}[htbp]
    \centering
    \begin{tikzpicture}
        \draw (-0.1,0.7) -- (0.7,-0.1);
        \draw[dotted] (0.5,-0.1) -- (1.3,0.7);
        \draw (1.1,0.7) -- (1.9,-0.1);
        \draw (1.65,0) -- (2.55,0);
        \draw (2.3,-0.1) -- (3.1,0.7);
        \draw[dotted] (2.9,0.7) -- (3.7,-0.1);
        \draw (3.5,-0.1) -- (4.3,0.7);
        \draw (0,0.45) -- (0,1.2);
        \draw[dotted] (0,0.9) -- (0,1.5);
        \draw (0.3,0.15) -- (0.3,1.1);
        \draw (0.4,1.2) arc (90:180:0.1);
        \draw (0.4,1.2) -- (3.8,1.2);
        \draw (3.9,1.1) arc (0:90:0.1);
        \draw (3.9,1.1) -- (3.9,0.15);
        
        \begin{scriptsize}
        \draw (0.15,0.05) node {$G_t$};
        \draw (1.4,0) node {$G_1$};
        \draw (2.1,0.2) node {$D$};
        \draw (2.8,0) node {$F_1$};
        \draw (4.0,0) node {$F_s$};
        \draw (-0.2,0.9) node {$B$};
        \draw (2.1,1.35) node {$\Gamma$};
        \end{scriptsize}
    \end{tikzpicture}
    \caption{Configuration on $\tilde{Z}$.}
    \label{de2WaW1W}
\end{figure}
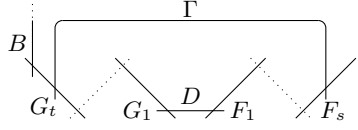
\end{lemma}

\begin{proof}
Let us assume we have such a configuration of curves. Let $f_1$, $\dots$, $f_s$, $g_1$, $\dots$, $g_t$ be so that $F_i^2=-f_i, G_j^2=-g_j$. We have $g_t \geq 3$, because otherwise by contracting $B$ y $\Gamma$ the curve $G_t$ is a $\P^1$ with $G_t^2=0$, but $\tilde{Z}$ is not ruled. If $f_s>2$, then $K_Z \cdot \Gamma >0$ because of the discrepancies. So, let us assume $f_s=2$. In particular $s \geq 2$. 

Say that $t=1$, and so $g_t=4$. Then if we contract $B, D,\Gamma$, we have $G_t^2=-1$. But also $F_s^2=-1$, and they intersect, a contradiction with $\tilde{Z}$ not ruled. Therefore we have $t \geq 2$, and with that $g_1=2$.

Note that in all contractions below, we can never have a chain $[1,2,\ldots,2,1]$ by the not ruled hypothesis.

Let us denote by $p:=f_1, q:=g_t$, and $\{k_1, \dots, k_s\}, \{l_1, \dots, l_t\}$ the corresponding discrepancies. We have the following four cases: 

\begin{itemize}
\item If $[f_1, \dots, f_s]$ and $[g_t, \dots, g_1]$ are of type $B$, then  there are  $p-2$ $(-2)$-curves starting with $F_s$. Because of the change of self-intersection of $G_t$ after contracting $B$, $\Gamma$ and the $p-2$ $(-2)$-curves, we have $q \geq p+1$. By Lemma \ref{cotadiscrepancias}, we have  \[ K \cdot \Gamma = -1 - k_s - l_t > -1 + \frac{1}{p} + 1 - \frac{1}{q-1} = \frac{q-p-1}{p(q-1)} \geq 0, \]

\item If $[f_1, \dots, f_s]$ is of type $B$ and $[g_t, \dots, g_1]$ is of type $M$, we have $p-2$ $(-2)$-curves starting with $F_s$, and so $q \geq p+1$ just as before. But in addition we can blow-down $D, G_1,\ldots, G_{t-1}$ which gives the better restriction $q \geq p+2$. Lemma \ref{cotadiscrepancias} gives in this case \[ K \cdot \Gamma = -1 - k_s - l_t > -1 + \frac{1}{p} + 1 - \frac{1}{q-2} = \frac{q-p-2}{p(q-2)} \geq 0 .\]

\item If $[f_1, \dots, f_s]$ is of type $M$ and $[g_t, \dots, g_1]$ is of type $B$, then we have $p-4$ $(-2)$-curves starting with $F_s$, and so $q\geq p-1$. This implies \[ K \cdot \Gamma = -1 - k_s - l_t > -1 + \frac{1}{p-2} + 1 - \frac{1}{q-1} = \frac{q+1-p}{(p-2)(q-1)} \geq 0. \]

\item If both $[f_1, \dots, f_s]$ and $[g_t, \dots, g_1]$ are of type $M$, we obtain $q \geq p-1$ just as done before.  Contracting $D$ and $G_1, \dots, G_{t-1}$, we get $q > p-1$. We also have \[ K \cdot D = -1 + k_1 + l_1 = -\frac{1}{p-2}+\frac{1}{q-2}=\frac{p-q}{(p-2)(q-2)}>0, \]
and so $p>q$. But then $p>q>p-1$ which is not possible.
\end{itemize}
\end{proof}

We now finish the proof of Theorem \ref{2W}.

Let $Z_1:=X_2$. If $K_{Z_1}$ is nef, then we are done. If not, then by Lemma \ref{-2} we must have a mk1A as in Lemma \ref{controlcurva2Wahl}. Using Lemma \ref{mk1Aesflip}, we can now apply the flip and get $Z_2$ that sits in two possible situations: Type(-1) or Type(-2).

We now assume that we have a chain of flips giving only Type(-2), or a chain of Type(-2) followed by one Type(-1).

If only Type(-2), then by Lemma \ref{-2} we have that the new mk1A nbhd can only be as in Lemma \ref{controlcurva2Wahl}, and we continue, or $K$ is nef.

If only Type(-2) and one last Type(-1), then we cannot have a mk1A neighborhood by Lemma \ref{-1}. And so either $K$ is nef, or we have a mk2A nbhd. Then by Lemma \ref{prelemacompartido} we can only have a $\Gamma$ intersecting $G_1,F_1$ (not possible since $K \cdot D >0$), or $F_1, G_t$ or $G_1, F_s$. Note that in both cases we have $t>1$ or $s>1$ since otherwise we are can use Lemma \ref{prelemacompartido} or that $K \cdot D >0$. Therefore we can apply Proposition \ref{propcompartida} to deduce that a mk2A nbhd is impossible, and so $K$ must be nef. This process must end in finitely many steps, so we are done.

\section{Open questions} \label{s6}

\subsection{Topological type of surfaces in a wormhole} Let us start with a couple of examples. Consider a general rational elliptic surface $Z \to \P^1$ with sections and no $(-1)$-curves in their fibers. Hence any section is a $(-1)$-curve. Let $F_E$ and $F_G$ two nodal $I_1$ fibers, and let $S$ be a section of $Z \to \P^1$. We blow-up $s$ times over the node in $F_E$, and $r$ times over the node in $F_G$ to obtain a surface $\tilde X_1$ with the Wahl chains $[F_E,E_1,\ldots,E_{s-1}]=[3+s,2,\ldots,2]$ and $[F_G,G_1,\ldots,G_{r-1}]=[3+r,2,\ldots,2]$. The contraction of both of them produces a W-surface $X_1$ (see \cite[Theorem 4.2]{U16b}), and the general fibers are either Enriques surfaces (if $r=s=1$) or elliptic surfaces of Kodaira dimension $1$. In fact, one can prove that the general fiber is an elliptic fibration over $\P^1$ with $p_g=q=0$ and two multiple fibers of multiplicities $s$ and $r$, and so its fundamental group is $\Z/\text{gcd}(r,s)$. Hence, although these are not degenerations of surfaces of general type, they will be useful to see that wormholes may change the topology of the general fibers.

The curve $S$ defines an extremal P-resolution on $X_1$. Let us consider the chain of curves $E_{s-1},\ldots,E_1,F_E,S,F_G,G_1,\ldots,G_{r-1}$. Their contraction defines the cyclic quotient singularity $(Q \in Y)$ given by $$[2,\ldots,2,2+s,2+r,2,\ldots,2],$$ whose dual continued fraction is $[s+1,\bar{2},\ldots,2,3,2,\ldots,\bar{2},r+1]$, where the numbers of $2$'s are $s-1$ (left) and $r-1$ (right), and we mark with bars the position of the pair which produces the extremal P-resolution indicated above. (The cases $s=1$ or $r=1$ are a bit different as the reader may check.) We want to check whether $(Q \in Y)$ is a wormhole singularity, and so we are looking for another pair. A quick verification shows that $r>3$ or $s>3$ do not work. For the few cases left and up to reordering, the only wormhole singularities are $[4,\bar2,\underline{2},3,\bar2,\underline{3}]$ and $[\underline{4},\bar2,\underline{2},\bar3,2]$, corresponding to the initial extremal P-resolutions (I) $[2,2,6]-1-[5,2]$ and (II) $[2,2,6]-1-[4]$ respectively. We have:  
\bigskip

\textbf{(I).} In this case, the new extremal P-resolution is $[2,2,5,4]-2$. Let $X_2$ be the corresponding W-surface. The curve $G_2$ is now a flipping curve, and after the flip we obtain a W-surface $X_2'$ with extremal P-resolution $[2,2,6]-1-[4]$. Therefore the canonical class now is nef. The general fiber of $X_1$ gives an elliptic surface with fundamental group of order gcd$(4,3)=1$, but the general fiber of $X_2'$ has fundamental group of order gcd$(4,2)=2$. Thus they are not homeomorphic.
\bigskip

\textbf{(II).} In this case, the new extremal P-resolution is $2-[2,5,3]$. Let $X_2$ be the corresponding W-surface. The curve $E_3$ is now a flipping curve, and after the flip we obtain a W-surface $X_2'$ with extremal P-resolution $[2,5]-1-[4]$, and so the canonical class now is nef. The general fiber of $X_1$ gives an elliptic surface with fundamental group of order gcd$(4,2)=2$, but the general fiber of $X_2'$ is simply connected, since gcd$(3,2)=1$. Thus they are not homeomorphic as well.

\bigskip
However, in many cases wormholes produce surfaces with isomorphic fundamental groups. Let us consider a wormhole situation from $X_1$ to $X_2$, where both have two Wahl singularities corresponding to the extremal P-resolutions. Let $d_i$ be the greatest common divisor of the indices of the Wahl singularities in $X_i$. (If there is one or zero singularities, then $d_i=1$.)

\begin{proposition}
If $d_1=d_2$, then the fundamental groups of the general fiber of $X_1$ and $X_2$ are isomorphic.
\label{d1=d2}
\end{proposition}

\begin{proof}
Let $f_i^+ \colon (C_i \subset X_i) \to (Q \in Y)$ be the contractions to a wormhole singularity. We are going to use the Seifert--Van-Kampen theorem to compare the fundamental groups of the general fibers $X_{1,t}$ and $X_{2,t}$. Let $L$ be the link of $(Q \in Y)$. Let $M_i$ be the Milnor fiber of the smoothing of $(Q \in Y)$ corresponding to $X_i$ (i.e. the blowing-down deformation of the $\Q$-Goresntein smoothing corresponding to the extremal P-resolution in $X_i$). Then $\pi_1(M_i) \simeq \Z/d_i$, and $\pi_1(L) \simeq \Z/\Delta$ where $\frac{1}{\Delta}(1,\Omega) = (Q \in Y)$. Let $X_i^0$ be the complement of $C_i$. Then $X_1^0=X_2^0=:X$, and we have $$\pi_1(X_{i,t}) \simeq \big(\pi_1(X) \star \pi_1(M_i) \big)/N(\alpha \beta^{-1})$$ where $\alpha$ generates $\pi_1(L)$ in $\pi_1(X)$, $\beta$ generates $\pi_1(L)$ in $\pi_1(M_i)$, and $N(\alpha \beta^{-1})$ is the smallest normal subgroup containing $\alpha \beta^{-1}$. By \cite[Lemma 5.1]{LW86}, we have that the morphism induced by the inclusion $\pi_1(L) \to \pi_1(M_i)$ is onto. Therefore, if $\pi_1(X)=G/R$, where $G$ are generators and $R$ are relations, then $\pi_1(X_{i,t}) \simeq G/(R,\alpha^{d_i}=1)$. The claim follows when $d_1=d_2$.
\end{proof}

\begin{corollary}
If $d_1=d_2=1$, then the general fibers of $X_1$ and $X_2$ have isomorphic fundamental groups and equal to $\pi_1(\tilde X_1)=\pi_1(\tilde X_2)=\pi_1(Y)$. In particular, if in addition $\tilde X_1$ is rational, then wormholes produce simply connected surfaces. 
\end{corollary}

\begin{proof}
Here, by applying the Seifert--Van-Kampen theorem, we have that $\pi_1(X_i)=\pi_1(X_i^0)/(\alpha=1)$, but this is what we just computed for $\pi_1(X_{i,t})$ when $d_i=1$ (alternatively one can use \cite[Theorem 3]{LP07}). The other claim is because we are dealing with rational singularities. 
\end{proof}

Let us consider rational W-surfaces $X_i$ with $d_1=d_2=1$. Let us assume $K^2=1$, and so their general fibers are simply connected Godeaux surfaces. There are plenty of such wormholes in the KSBA compactification of the moduli space of Godeaux surfaces (see e.g. \cite[Fig. 6]{LP07} for the $[2,2,6]-1-[4]$). By Freedman's classification theorem, the general fibers are homeomorphic as oriented 4-manifolds. On the other hand, Miles Reid conjectures that the moduli space of torsion zero Godeaux surfaces is irreducible, and so all of these wormholed surfaces should be diffeomorphic. Very recently, Dias and Rito proved Reid's conjecture for $\Z/2$-Godeaux surfaces in \cite{DR20}, and so any wormhole in their KSBA compactification with $d_1=d_2$ (as in Example \ref{ejenriques}) gives diffeomorphic surfaces.       

\begin{question}
For a wormhole with $d_1=d_2$, are the general fibers always diffeomorphic? homeomorphic?
\end{question}


In fact one can show that $d_1=d_2$ keeps the homology together with the intersection form, and so if they are simply connected, then Freedman's theorem produces an homeomorphism. More on the topology aspects will be part of a sequel work. On the other hand and as we saw above, for the case $d_1 \neq d_2$ we may have non homeomorphic surfaces (although the example was not of general type). In \cite[Figure 5]{DRU20} we have a wormhole defined by $[2,5]-1-[2,6,2,3]$ in $X_1$ ($d_1=1$), whose $\Q$-Gorenstein smoothing is a $\Z/2$-Godeaux surface. Its wormholed surface $X_2$ has extremal P-resolution $[2,3,5,3]-1-[4]$, and so $d_2=2$. If its $\Q$-Gorenstein smoothing has $\pi_1 \neq \Z/2$, then it would be $\Z/4$ by the classification we have for Godeaux surfaces.


We note that for wormholes of general type and with different fundamental groups, we would be crossing distinct components of the 
moduli space. We expect there are many. 



\subsection{What is left to prove the conjecture} In this paper we introduced the wormhole conjecture, and we proved it for many situations under the assumption that the singular surfaces involved were not rational. Hence we divide the final discussion in two parts:

\bigskip
\noindent
\textbf{Nonrational:} Let $X_1$ and $X_2$ be the W-surfaces in a wormhole, both with an extremal P-resolution over a fixed wormhole singularity, and nonsingular out of them. In the next list, we write Wahl-$m$-Wahl for an extremal P-resolution with two Wahl singularities (distinct or equal) and a middle curve whose self-intersection in the minimal resolution is $-m$. If Wahl is dropped, then the point is nonsingular. Using Theorem \ref{calcularc}, and because we already have Theorems \ref{W1W} and \ref{2W}, the list of pairs of extremal P-resolutions where we do not know the validity of the wormhole conjecture is:

\begin{itemize}
    \item[(a)] Wahl-$m$-Wahl and Wahl-$m$-Wahl for $m \geq 2$.
    
    \item[(b)] $m$-Wahl and Wahl-$(m-1)$-Wahl for $m \geq 3$.
    
    \item[(c)] Wahl-$m$ and $m$-Wahl for $m \geq 3$.
\end{itemize}

For the case (c) we will give some combinatorial counterexamples to the wormhole conjecture, although we do not know if they can be realized on a surface.

\begin{example} \label{counterex}
Let us assume the existence of a chain of $\P^1$'s $E_1, \ldots, E_9$ in a nonsingular surface $Z$ with nef minimal model, where  $E_i^2=-e_i$ , and 
$$[e_1, \ldots, e_9] = [5, 2, 2, 2, 8, 2, 2, 2, 5].$$ Assume that there is a $(-1)$-curve $\Gamma$ intersecting $E_1$ twice and transversally, and disjoint from the rest. This $(-1)$-curve does not produce any contradiction with the minimal model of $Z$. The wormhole singularity $[5, 2, 2, 2, 10, 2, 2, 2, 5]$ admits two obvious extremal P-resolutions: $[5, 2, 2, 2, 10, 2, 2, 2]-5$ in $X_1$ and $5-[2, 2, 2, 10, 2, 2, 2, 5]$ in $X_2$, so we are in case (c). The curve $\Gamma$ is positive for $K_{X_1}$, and it is not only negative for $K_{X_2}$ but it induces a divisorial contraction on the deformation of $X_2$. 
In fact $[n+2, \underbrace{2, \dots, 2}_{n}, n+5, \underbrace{2, \dots, 2}_{n}, n+2]$ with a $(-1)$-curve intersecting only the first curve with multiplicity $n-1$ gives infinitely many bad situations. Are any of these counterexamples realizable?
\end{example}

\noindent
\textbf{Rational:} Here we do not have a feasible strategy to prove the conjecture. But we have many examples verifying it for the invariants $p_g=q=0$ and $K^2=1,2,3,4$. These examples are constructed as in \cite{LP07} and they have two singularities, they will be part of some future work.   

\bigskip
We finish the paper with another open question. Note that a WW singularity (i.e. it admits at least one pair of indices to be a zero continued fraction) has complete freedom on the values of $\delta$. But this freedom is lost for wormhole singularities.

\begin{question}
What are the possible values for $\delta$ in a wormhole singularity?
\end{question}

For $\Delta \leq 450$ we have only $64$ wormhole singularities $\frac{1}{\Delta}(1,\Omega)$, and the values of $\delta$ are $2, 5, 10, 13, 17, 26, 30, 37, 50$. These values appear with multiplicities $31, 18, 4, 3, 3, 1, 2, 1, 1$ respectively. If we consider wormholes singularities whose Hirzebruch--Jung continued fraction has at most $18$ entries and their extremal P-resolutions requires no blow-ups, then the values of $\delta$ are: \begin{gather*}
2, 5, 10, 13, 17, 26, 30, 34, 37, 50, 53, 58, 65, 68, 82, 89, 101, 122, 130, \\
145, 170, 178, 185, 197, 219, 222, 226, 233, 257, 290, 317, 325, 327, 338, \\
350, 457, 466, 520, 578, 610, 738, 853, 964, 986, 997, 1010, 1220, 1237, \\
1342, 1515, 1597, 1740, 1970, 2018, 2210, 2487, 2758, 3005, 3194, 3390, \\
3505, 3567, 4112, 4181, 4930, 5722, 5725, 5850, 6878, 9282.
\end{gather*}

We also have that 
\begin{gather*}
\delta = 2, 5, 10, 13, 17, 26, 30, 34, 37, 50, 53, 58, 65, 68, 82, 89, 101, 122, 130, \\
145, 178, 185, 219, 222, 233, 317, 327, 338, 350, 457, 466, 520, 578, 610, \\
738, 853, 964, 986, 997, 1010, 1220, 1237, 1342, 1515, 1597, 1740, 1970, \\
2018, 2210, 2487, 2758, 3005, 3194, 3390, 3505, 3567, 4112, 4181, 4930, \\
5722, 5725, 5850, 6305, 6878, 7298, 8020, 9282, 10670, 10946, 11482, \\
12190, 13669, 13848, 15049, 15650, 17602, 19710, 20917, 24418, 27030, \\
28657, 29822, 39338, 75025
\end{gather*}
are all the values of $\delta$ for wormholes singularities whose Hirzebruch-Jung continued fraction has at most $25$ entries, and with one of their extremal P-resolutions being of type $m$-Wahl with $m=2$. 

The following infinite family has the value $\delta=2$: 
$$ \frac{\Delta}{\Omega}= 2-[\underbrace{2, \dots, 2}_{k-2}, 5, k] = [\underbrace{2, \dots, 2}_{k-1}, k+3] - 1 - [\underbrace{2, \dots, 2}_{k-3}, k+1]. $$
One can compute
$$ [\underbrace{2, \dots, 2}_{k-2}, 5, k]=\frac{(2k-1)^2}{(2k-1)(2k-3)-1}, $$
which gives $\Delta=4k^2, \Omega=(2k-1)^2, \delta=2$. Actually, the case $\delta=2$ can be completely classified through the use of triangulations of polygons, where one changes one diagonal in a ``corner quadrilateral" by the other diagonal.

On the other hand, not every natural number appears as the $\delta$ of a wormhole singularity. For instance $\delta=3$ is not possible. Indeed, say that $\frac{\Delta}{\Delta-\Omega}=[a_1, \dots, \overline{a}_\alpha, \dots, \overline{a}_\beta, \dots, a_r]$ has $\delta=3$. We may assume also that the $a_0>1$. Thus, $a_\alpha=a_\beta=2, \frac{\delta}{\varepsilon}=[a_{\alpha+1}, \dots, a_{\beta-1}]$, and $\frac{\delta}{\delta-\varepsilon}=[a_{\alpha-1}, \dots, a_1, a_0, a_r, \dots, a_{\beta+1}]$ from the case (A). Now, we have $\varepsilon=1$ or $2$, so one of those continued fractions is $\frac{3}{2}=[2, 2]$. We get that the associated triangulation contains $\{\overline{2}, 2, 2, \overline{2}\}$, which is clearly a contradiction.

A better understanding of the wormhole phenomenon on singularities is wanted, to potentially solve the wormhole conjecture and to show topological implications.


\end{document}